\documentclass[10.5pt]{amsart}
\usepackage{pgf,tikz,pgfplots,color}
\pgfplotsset{compat=1.15}
\usepackage{mathrsfs} %mathscr
\usetikzlibrary{arrows}
\usepackage{fancyhdr}
\usepackage{lastpage} 

\usepackage[margin=1in]{geometry}
\usepackage{geometry}
\usepackage{hyperref}
\hypersetup{
    colorlinks=true,
    linkcolor=blue,
    citecolor=brown,
    filecolor=magenta,
    urlcolor=blue,
}
\usepackage{amsmath}
\usepackage{amsfonts}
\usepackage{amssymb}
\usepackage{amsthm}
\usepackage{setspace}
\usepackage{inputenc}
\usepackage{comment}
\usepackage{mathtools}
\usepackage[shortlabels]{enumitem}

\newtheorem{thm}{Theorem}
\newtheorem{prop}[thm]{Proposition}
\newtheorem{lem}[thm]{Lemma}
\theoremstyle{definition}
\newtheorem{defn}[thm]{Definition}

\newtheorem{innercustomgeneric}{\customgenericname}
\providecommand{\customgenericname}{}
\newcommand{\newcustomtheorem}[2]{%
  \newenvironment{#1}[1]
  {%
   \renewcommand\customgenericname{#2}%
   \renewcommand\theinnercustomgeneric{##1}%
   \innercustomgeneric
  }
  {\endinnercustomgeneric}
}

\newcustomtheorem{customprop}{Proposition}
\newcustomtheorem{customlemma}{Lemma}

\theoremstyle{remark}
\newtheorem{rem}[thm]{Remark}

\newcommand{\Z}{\mathbb{Z}}
\newcommand{\R}{\mathbb{R}} 
\newcommand{\C}{\mathbb{C}}
\newcommand{\N}{\mathbb{N}}
\newcommand{\1}{\mathbf{1}}

\newcommand{\As}{\mathscr{A}}
\newcommand{\Bs}{\mathscr{B}}
\newcommand{\Ds}{\mathscr{D}}
\newcommand{\Ec}{\mathcal{E}}
\newcommand{\Eb}{\mathbf{E}}
\newcommand{\Es}{\mathscr{E}}
\newcommand{\Fs}{\mathscr{F}}
\newcommand{\Ns}{\mathscr{N}}

\newcommand{\Ss}{\mathscr{S}}
\newcommand{\Uc}{\mathcal{U}}
\newcommand{\Xs}{\mathscr{X}}

\newcommand{\eps}{\varepsilon}
\newcommand{\supp}{\operatorname{supp}}

\newcommand{\loc}{\mathrm{loc}}

\newcommand{\Coorvec}[1]{\frac\partial{\partial#1}}

\def\XXint#1#2#3{{\setbox0=\hbox{$#1{#2#3}{\int}$ }
\vcenter{\hbox{$#2#3$ }}\kern-.6\wd0}}

	\title[Half space extension]{On Seeley-type Universal Extension Operators for the Upper Half Space}        
\author[]{Haowen Lu} 
\address{Haowen Lu, Department of Mathematics,
	The Chinese University of Hong Kong, Shatin, N.T., Hong Kong, China} 
\email{hwlu@math.cuhk.edu.hk}
\author[]{Liding Yao} 
\address{Liding Yao, Department of Mathematics,
	The Ohio State University, Columbus, OH 43210, United States} 
\email{yao.1015@osu.edu}


\begin{document}
\begin{abstract}
    Modified from the standard half-space extension via reflection principle, we construct a linear extension operator for the upper half space $\Bbb R^n_+$ that has the form $Ef(x)=\sum_{j=-\infty}^\infty a_jf(x',-b_jx_n)$ for $x_n<0$. We prove that $E$ is bounded in all $C^k$-spaces, Sobolev and H\"older spaces, Besov and Triebel-Lizorkin spaces, along with their Morrey generalizations. We also give an analogous construction on bounded smooth domains.
\end{abstract}
\maketitle
\vspace{-0.15in}
\section{Introduction}
Given a function $f$ defined on the upper half space $\R^n_+=\{(x',x_n):x_n>0\}$, the  \textit{reflection principle} gives a well-known construction to extend $f$ to the total space $\R^n$ while preserving its regularity property: we can define the extension $Ef$ that depends linearly on $f$ by
\begin{equation}\label{Eqn::ExtOpIntro}
    \textstyle Ef(x',x_n)=E^{a,b}f(x',x_n):=\sum_ja_jf(x',-b_jx_n),\quad\text{for }x_n<0,\quad\text{where }b_j>0.
\end{equation}
Here $(a_j,b_j)_j$ are  (finitely or infinitely many) real numbers that satisfy certain algebraic conditions (see below). This method traces back to Lichtenstein \cite{LichtensteinExtension} for $C^1$-extensions and \cite{HestenesExtension} for $C^k$-extensions. See also \cite{Nikolskii} and \cite{Babic} for the corresponding Sobolev extensions.
% They showed that if $a_j,b_j$ are finite and $\sum_ja_j(-b_j)^k=1$ for $k=0,1,\dots,m$, then $E:C^k(\overline{\R^n_+})\to C^k(\R^n)$ is bounded linear for $0\le k\le m$.

In the spaces of distributions, a linear extension operator $E$ is viewed as a right inverse of the natural restriction map $\Ds'(\R^n)\twoheadrightarrow\Ds'(\R^n_+)$. Triebel \cite{TriebelInterpolation} and Franke \cite{FrankeDomain} showed that for every $\eps>0$ there is a $m=m(\eps)>0$, such that, if the finite collection $(a_j,b_j)_j$ satisfy $\sum_ja_j(-b_j)^k=1$ for all integers $-m\le k\le m$, then for $E$ in \eqref{Eqn::ExtOpIntro} have boundedness in Besov spaces $E:\Bs_{pq}^s(\R^n_+)\to\Bs_{pq}^s(\R^n) $ and in Triebel-Lizorkin spaces $E:\Fs_{pq}^s(\R^n_+)\to\Fs_{pq}^s(\R^n)$, for all $\eps<p,q\le\infty$ and $-\eps^{-1}<s<\eps^{-1}$ ($p<\infty$ in $\Fs$-cases). In \cite[Section~1.11.5]{TriebelTheoryOfFunctionSpacesIII} such $E$ is called a \textit{common extension operator} which depends on $\eps(>0)$. 

In contrast, we call $E$ a \textbf{universal extension operator}, if we have the boundedness $E:\As_{pq}^s(\R^n_+)\to\As_{pq}^s(\R^n)$ for all $\As\in\{\Bs,\Fs\}$, $0<p,q\le\infty$ and $s\in\R$ ($p<\infty$ in $\Fs$-cases), simultaneously (i.e. we can take $\eps=0$).

Note that the collections $(a_j,b_j)_j$ discussed above are all finite. In those cases, the range of spaces on which $E$ is bounded would inevitably be finite and depends on $\eps>0$. Hence such $E$ is never universal. 

By taking some infinitely nonzero sequences $(a_j,b_j)_j$, Seeley \cite{Seeley} constructed one such operator \eqref{Eqn::ExtOpIntro} such that $E:C^k(\R^n_+)\to C^k(\R^n)$ for all $k=0,1,2,\dots$. In particular Seeley's extension  perverse $C^\infty$-smoothness. However, to the best of authors' knowledge, there was no proof of the boundedness of Seeley's operator on general Besov and Triebel-Lizorkin spaces.

In this paper we generalize Seeley's construction by extending the boundedness to the spaces of negative index, and we show that such operator is a universal extension operator. %This appears to be a well-known result, but no proof is found in the literature, to the best of the authors' knowledge.
\begin{thm}\label{MainThm}
\begin{enumerate}[(i)]
    \item\label{Item::Thm::Exist} There exists an extension operator $E:C^0(\overline{\R^n_+})\to C^0(\R^n)$  that has the form
\begin{equation}\label{Eqn::ExtOp}
    Ef(x',x_n)=E^{a,b}f(x',x_n):=\begin{cases}\sum_{j=-\infty}^\infty a_jf(x',-b_jx_n)&x_n<0
    \\
    f(x)&x_n>0
    \end{cases},\quad\text{where }b_j>0\text{ such that},
\end{equation}

\vspace{-0.12in}
\begin{equation}\label{Eqn::CondAB}
\text{for every }k\in\Z,\ \sum_{j=-\infty}^\infty a_j(-b_j)^k=1\text{ and }\sum_{j\in\Z}2^{\delta_k|j|} |a_j|\cdot b_j^k<\infty\text{ for some }\delta_k>0.
\end{equation}
\end{enumerate}Moreover, when \eqref{Eqn::CondAB} is satisfied, then \eqref{Eqn::ExtOp} always defines an extension operator $E$ for functions on $\R^n_+$, that has the following boundedness (simultaneously):
\begin{enumerate}[(i)]\setcounter{enumi}{1}
    \item{\normalfont(Sobolev, H\"older and $C^k$)}\label{Item::Thm::SobBdd} $E:W^{k,p}(\R^n_+)\to W^{k,p}(\R^n)$ and $E:C^{k,s}(\R^n_+)\to C^{k,s}(\R^n)$ are defined and bounded for all $k\in\Z$, $0<p\le\infty$ and $0<s<1$.
    
     $E:C^k(\R^n_+)\to C^k(\R^n)$ is bounded for all $k=0,1,2,\dots$.

    \item{\normalfont(Besov and Triebel-Lizorkin)}\label{Item::Thm::BFBdd} $E:\Ss'(\R^n_+)\to\Ss'(\R^n)$ is continuous. Moreover we have boundedness in Besov spaces $E:\Bs_{pq}^s(\R^n_+)\to \Bs_{pq}^s(\R^n)$ and in Triebel-Lizorkin spaces $E:\Fs_{pq}^s(\R^n_+)\to \Fs_{pq}^s(\R^n)$ for all $s\in\R$ and $0<p,q\le\infty$.
    \item{\normalfont(Morrey-type)}\label{Item::Thm::MorreyBdd} More generally $E$ has boundedness on all Besov-type spaces $\Bs_{pq}^{s\tau}$, Triebel-Lizorkin-type spaces $\Fs_{pq}^{s\tau}$ and Besov-Morrey spaces $\Ns_{pq}^{s\tau}$. That is, $E:\As_{pq}^{s\tau}(\R^n_+)\to \As_{pq}^{s\tau}(\R^n)$ for $\As\in\{\Bs,\Fs,\Ns\}$ and all $0<p,q\le\infty$, $s\in\R$ and $0\le \tau\le\frac1p$ ($p<\infty$ for the $\Fs$-cases).
\end{enumerate}
\end{thm}
See Definitions \ref{Defn::Sob}, \ref{Defn::Ss} and \ref{Defn::BsFs} for the spaces in the theorem.
Here the Sobolev spaces $W^{k,p}$ for $p<1$ can be defined and are discussed in \cite{PeetreSobolev}, see also Remark~\ref{Rmk::SobRmk1}.

The analogue of Theorem~\ref{MainThm} on smooth domains is also true; see Theorem~\ref{Thm::Domain}. The operator is given by \eqref{Eqn::DefEc}.

\begin{rem}\label{Rmk::DeltaK}
    Here we allow $\delta_k$ in \eqref{Eqn::CondAB} to tend to $0$ as $|k|\to\infty$. In practice we can choose $(a_j,b_j)_j$ such that the sum $\sum_j2^{\delta|j|}|a_j|b_j^k<\infty$ holds for all $k$ and all arbitrarily \textit{large} $\delta$, see Proposition \ref{Prop::ConsAB}.
    
    It is not known to the authors whether the results still hold if we remove $2^{\delta_k|j|}$ in \eqref{Eqn::CondAB}, i.e. if we only assume $\sum_j|a_j|b_j^k<\infty$ for all $k\in\Z$. In the proof of Theorem \ref{MainThm} the terms $2^{\delta_k|j|}$ are used only when we consider $p<1$ in \ref{Item::Thm::SobBdd} and $\min(p,q)<1$ in \ref{Item::Thm::BFBdd} and \ref{Item::Thm::MorreyBdd}.
\end{rem}

\medskip
We should remark that the universal extension operator exists not only on smooth domains but also on Lipschitz domains. This is done by Rychkov \cite{RychkovExtension} using Littlewood-Paley decompositions. Rychkov's extension operator $E^R$ has the form $E^Rf=\sum_{j=0}^\infty\psi_j\ast(\1_{\R^n_+}\cdot(\phi_j\ast f))$ where $\1_{\R^n_+}$ is the characteristic function of $\R^n_+$ and $(\phi_j,\psi_j)_{j=0}^\infty$ is a carefully chosen family of Schwartz functions. The shape of $E^R$ is totally different from \eqref{Eqn::ExtOpIntro} and the construction follows a different methodology.

Our result shows that the construction \eqref{Eqn::ExtOpIntro} can also be made to be universal extension. Cf. the comment below \cite[Corollary~5.7]{JohnsenHansenSickel}.

In fact, our extension operator is ``more universal'' than the Rychkov extension $E^R$ \cite{RychkovExtension}. Although $E^R$ is bounded in all Besov and Triebel-Lizorkin spaces (see \cite[Theorem~4.1]{RychkovExtension}), it is not known whether $E^R$ is bounded on the endpoint Sobolev spaces $W^{k,1}$, $W^{k,\infty}$ and the $C^k$-spaces. On the other hand, our extension operator is also defined on the measurable function space $L^p(\R^n_+)$ for $0<p<1$, whose element may not be realized as a distribution on $\R^n_+$.

\smallskip
The range of the extension operator is important if we want to construct some operator that has the form
\begin{equation}\label{Eqn::AnIntOp}
    \textstyle Tf(x)=\int_\Uc K(x,y)\,\Ec f(y)dy,\quad\text{where $\Ec$ is an extension operator for the domain $\Omega\subset\subset\Uc$.}
\end{equation}
Here $K(x,y)$ can be some kernel of integral operator or singular integral operator. 

This method has been used extensively in solving Cauchy-Riemann problem (the $\overline{\partial}$-equation) on certain domains in $\C^n$. For example Wu \cite{ADbarPaper} used the Seeley's extension to construct a solution operator that is $W^{k,p}$ bounded for \textbf{all} $k\ge0$. If $\Ec$  is some common extension operator that is only bounded in a finite range of $k$, then $T$ is only bounded in a finite range as well. 

By replacing the Seeley's extension with our extension operator and doing more analysis on $K(x,y)$, it might be possible that one can show that $T$ in \eqref{Eqn::AnIntOp} is bounded in $W^{k,p}$ for $k<0$ as well. 

The estimates on negative Sobolev spaces are discussed in \cite{ShiYaoCk,YaoFiniteType}, where we took $\Ec$ to be Rychkov's extension operator. In the case of smooth domains, if we replace Rychkov's extension by our Seeley-type extension, then it is possible not only to get some estimates related to $L^1$-Sobolev spaces, but also to simplify some of the proofs in these papers (see for example \cite[version 1, Proposition~5.1 and Remark~5.2 (iii)]{YaoFiniteType}). 

\section{Function Spaces and Notations}
In the paper we use the following definitions for Sobolev spaces and H\"older spaces, including both positive and negative indices. We also include the case $p<1$ for Sobolev spaces.
\begin{defn}\label{Defn::Sob}
Let $U\subseteq\R^n$ be an arbitrary open subset. Let $k\in\N_0=\{0,1,2,\dots\}$.

We use $C^k(U)=C^k(\overline{U})$ for the space of all continuous functions $f:U\to\R$ such that $\partial^\alpha f$ are bounded and uniformly continuous for all $|\alpha|\le k$. We use $\|f\|_{C^k(U)}:=\max_{|\alpha|\le k}\|\partial^\alpha f\|_{C^0(U)}$. 

We let $C^\infty(U)=C^\infty(\overline{U}):=\bigcap_{k=0}^\infty C^k(U)$ to be the space of bounded smooth functions.

For $0<s<1$, we define the H\"older space $C^{k,s}(U)$ to be the space of all functions $f\in C^k(U)$ such that $\|f\|_{C^{k,s}(U)}:=\|f\|_{C^k(U)}+\max_{|\alpha|\le k}\sup_{x,y\in U}|x-y|^{-s}|\partial^\alpha f(x)-\partial^\alpha f(y)|<\infty$. 

\smallskip
For $0<p\le\infty$, we define the Sobolev space $W^{k,p}(U)$ by the following:
\begin{enumerate}[(i)]
    \item\label{Item::Sob::p>1} For $1\le p\le\infty$, $W^{k,p}(U)$ consists of all $f\in L^p(U)$ whose distributional derivatives $\partial^\alpha f$ belong to $L^p(U)$ for all $|\alpha|\le k$. We use the norm $\|f\|_{W^{k,p}(U)}:=\big(\sum_{|\alpha|\le k}\int_U|\partial^\alpha f|^p\big)^\frac1p$ for $1\le p<\infty$ and $\|f\|_{W^{k,\infty}(U)}:=\sup_{|\alpha|\le k}\|\partial^\alpha f\|_{L^\infty(U)}$.
    \item\label{Item::Sob::p<1} For $0<p<1$, the quasi-Banach space $W^{k,p}(U)$ is the abstract completion\footnote{$C^\infty_c(\overline{U})\subseteq C^\infty(\overline{U})$ is the set of all bounded smooth functions $f:U\to\R$ such that $f|_{U\cap\{|x|>R\}}\equiv0$ for some $R$ (depending on $f$).} of $ C^\infty_c(\overline{U})$ under the quasi-norm $\|f\|_{W^{k,p}(U)}:=\big(\sum_{|\alpha|\le k}\int_U|\partial^\alpha f|^p\big)^\frac1p$.
\end{enumerate}

For $k>0$, $0<s<1$ and $1\le p\le\infty$, we define
\begin{align}
    \label{Eqn::NegHoldNorm}\|f\|_{C^{-k,s}(U)}&:=\inf\Big\{\max_{|\alpha|\le k}\|g_\alpha\|_{C^{0,s}(U)}:\{g_\alpha\}_{|\alpha|\le k}\subset C^{0,s}(U),\quad f=\sum_{|\alpha|\le k}\partial^\alpha g_\alpha\text{ as distributions}\Big\}.
\\
    \label{Eqn::NegSobNorm}\|f\|_{W^{-k,p}(U)}&:=\inf\Big\{\Big(\sum_{|\alpha|\le k}\|g_\alpha\|_{L^p(U)}^p\Big)^{1/p}:\{g_\alpha\}_{|\alpha|\le k}\subset L^p(U),\quad f=\sum_{|\alpha|\le k}\partial^\alpha g_\alpha\text{ as distributions}\Big\}.
\end{align}
\begin{enumerate}[(i)]\setcounter{enumi}{2}
    \item We define $C^{-k,s}(U)$ and $W^{-k,p}(U)$ ($1\le p\le\infty$) to be the subset of distributions such that the above norms are finite, respectively.
    \item For $0<p<1$, we define $W^{-k,p}(U)$ to be the abstract completion of $C_c^\infty(\overline{U})$ under the quasi-norm \eqref{Eqn::NegSobNorm} with $\{g_\alpha\}_\alpha\subset L^p(U)$ replacing by $\{g_\alpha\}_\alpha\subset C_c^\infty(\overline{U})$ (where we only consider $f\in C_c^\infty(\overline{U})$ on both sides).
\end{enumerate}

\end{defn}

\begin{rem}\label{Rmk::SobRmk1}
    When $0<p<1$ and $k\ge1$, the natural (continuous) mapping $\iota:W^{k,p}(U)\to L^p(U)$ is \textbf{not} injective. The proof can be modified from \cite[Proposition~3.1]{PeetreSobolev}. If one use $W^{k,p}(U)/\ker\iota$ instead, then the elements can always be realized as measurable functions (in $L^p(U)$).
    
    One can check that  Theorem \ref{MainThm} \ref{Item::Thm::SobBdd} implies the boundedness $E:W^{k,p}(\R^n_+)/\ker\iota_{\R^n_+}\to W^{k,p}(\R^n)/\ker\iota_{\R^n}$.
\end{rem}
\begin{rem}\label{Rmk::SobRmk2}

When $U$ satisfies nice boundary condition, particularly when $U\in\{\R^n_+,\R^n\}$ or when $U$ is a bounded smooth domain, we have the following:
\begin{enumerate}[(a)]
    \item Definition \ref{Defn::Sob} \ref{Item::Sob::p>1} and \ref{Item::Sob::p<1} coincide, in the sense that $C^\infty_c(\overline{U})$ is dense in $W^{k,p}(U)$ for $p\in[1,\infty)$. For the proof see \cite[Section 5.3]{EvansPDE} or \cite[Theorem~3.17]{Adams}. Taking the distribution derivatives, we get the density of $C^\infty_c(\overline{U})$ in $W^{-k,p}(U)$ for $p\in[1,\infty)$ as well.
    \item\label{Item::SobRmk::Dual} When $1\le p<\infty$, $W^{-k,p'}(U)$ is the dual of $W^{k,p}_0(U):=\overline{C_c^\infty(U)}^{\|\cdot\|_{W^{k,p}(U)}}(\subset W^{k,p}(U))$. See \cite[Theorem 3.9]{Adams} for example.
    \item For every $k\in\Z$, we have equivalent norms $\|f\|_{W^{k+1,p}(U)}\approx\|f\|_{W^{k,p}(U)}+\sum_{j=1}^n\|\partial_jf\|_{W^{k,p}(U)}$ for\\ $1<p<\infty$, and $\|f\|_{C^{k+1,s}(U)}\approx\|f\|_{C^{k,s}(U)}+\sum_{j=1}^n\|\partial_jf\|_{C^{k,s}(U)}$ for $0<s<1$.
    
    However they fail for $p=1,\infty$: in both cases we have $\|f\|_{L^p(U)}\gtrsim\|f\|_{W^{-1,p}(U)}+\sum_{j=1}^n\|\partial_jf\|_{W^{-1,p}(U)}$ but the converse inequalities are false. This can done by using \ref{Item::SobRmk::Dual} and \cite[Theorem 5]{DivNotSur}.
\end{enumerate}
\end{rem}

We use the standard convention for spaces of tempered distributions.

\begin{defn}\label{Defn::Ss}
We use $\Ss(\R^n)$ for the Schwartz space and $\Ss'(\R^n)$ for the space of tempered distributions.

For an arbitrary open subset $U\subseteq\R^n$, we define $\Ss'(U):=\{\tilde f|_U:\tilde f\in\Ss'(\R^n)\}(\subset\Ds'(U))$ to be the space of distributions in $U$ that has tempered distributional extension. 

We let $\Ss(U):=\{f\in\Ss(\R^n):f|_{\overline{U}^c}\equiv0\}(\subseteq\Ss(\R^n))$ to be the space of Schwartz functions supported in $\overline{U}$.
\end{defn}
\begin{rem}\label{Rmk::DefSs}
    \begin{enumerate}[(i)]
        \item \label{Item::DefSs::DualS}By definition $\Ss'(U)=\Ss'(\R^n)/\{\tilde f\in\Ss'(\R^n):\tilde f|_U=0\}$ is the quotient space. Since $\Ss(\R^n)$ and $\Ss'(\R^n)$ are dual to each other with respect to their standard topologies, we see that $\Ss'(U)=\Ss(U)'$. See also \cite[Theorem~2.10.5/1]{TriebelInterpolation}.
        \item\label{Item::DefSs::BddDom}We see that $\Ss'(U)$ is always the subspace of extendable distributions, i.e. $\Ss'(U)\subseteq\{\tilde f|_U:\tilde f\in\Ds'(U)\}$. They are equal when $U$ is a bounded domain: indeed, take a $\chi\in C_c^\infty(\R^n)$ be such that $\chi|_U\equiv1$, if $\tilde f\in\Ds'(\R^n)$ extends a $f\in\Ds'(U)$, then $\chi\tilde f\in\Ss'(\R^n)$ extends such $f$ as well.
    \end{enumerate}
    
\end{rem}

The Besov and Triebel-Lizorkin spaces, along with their Morrey analogies, can be defined using Littlewood-Paley decomposition. Note that we do not use these characterizations directly in the proof.
\begin{defn}\label{Defn::BsFs} Let $\lambda=(\lambda_j)_{j=0}^\infty$ be a sequence of Schwartz functions satisfying:
\begin{itemize}
    \item\label{Item::Space::LambdaFour} The Fourier transform $\hat\lambda_0(\xi)=\int_{\R^n}\lambda_0(x)2^{-2\pi ix\xi}dx$ satisfies $\supp\hat\lambda_0\subset\{|\xi|<2\}$, $\hat\lambda_0|_{\{|\xi|<1\}}\equiv1$.
    \item\label{Item::Space::LambdaScal}  $\lambda_j(x)=2^{jn}\lambda_0(2^jx)-2^{(j-1)n}\lambda_0(2^{j-1}x)$ for $j\ge1$.
\end{itemize}
Let $0< p,q\le\infty$, $s\in\R$ and $0\le\tau\le\frac1p$, we define the \textbf{Besov-type} norm $\|\cdot\|_{\Bs_{pq}^{s\tau}(\lambda)}$, the \textbf{Triebel-Lizorkin-type} norm $\|\cdot\|_{\Fs_{pq}^{s\tau}(\lambda)}$ ($p<\infty$) and the \textbf{Besov-Morrey} norm $\|\cdot\|_{\Ns_{pq}^{s\tau}(\lambda)}$ as follow:
\begin{align*}
\|f\|_{\Bs_{pq}^{s\tau}(\lambda)}&:=\sup_{x\in\R^n,J\in\Z}2^{nJ\tau}\bigg(\sum_{j=\max(0,J)}^\infty2^{jsq}\Big(\int_{B(x,2^{-J})}|\lambda_j\ast f(x)|^pdx\Big)^\frac qp\bigg)^\frac1q;
\\
    \|f\|_{\Fs_{pq}^{s\tau}(\lambda)}&:=\sup_{x\in\R^n,J\in\Z}2^{nJ\tau}\bigg(\int_{B(x,2^{-J})}\Big(\sum_{j=\max(0,J)}^\infty|2^{js}\lambda_j\ast f(x)|^q\Big)^\frac pqdx\bigg)^\frac1p,&\text{provided }p<\infty;
    \\
    \|f\|_{\Ns_{pq}^{s\tau}(\lambda)}&:=\bigg(\sum_{j=0}^\infty\sup_{x\in\R^n,J\in\Z}2^{nJ\tau q+jsq}\Big(\int_{B(x,2^{-J})}|\lambda_j\ast f(x)|^pdx\Big)^\frac qp\bigg)^\frac1q.
    %\|f\|_{\Fs_{\infty q}^s(\lambda)}&\textstyle:=\sup_{x\in\R^n,J\in\Z}2^{NJ\frac1q}\|(2^{js}\lambda_j\ast f)_{j=\max(0,J)}^\infty\|_{L^q(B(x,2^{-J});\ell^q)}&p=\infty.
\end{align*}
Here for $q=\infty$ we take natural modifications by replacing the $\ell^q$-sums with the supremums over $j$.

We define the corresponding spaces $\As_{pq}^{s\tau}(\R^n)=\{f\in\Ss'(\R^n):\|f\|_{\As_{pq}^s(\R^n)}:=\|f\|_{\As_{pq}^s(\lambda)}<\infty\}$ where $\As\in\{\Bs,\Fs,\Ns\}$, for a fixed choice of $\lambda$ ($p<\infty$ for the $\Fs$-cases).

For arbitrary open $U\subseteq\R^n$, we define $\As_{pq}^{s\tau}(U):=\big\{\tilde f|_U:\tilde f\in\As_{pq}^{s\tau}(\R^n)\big\}$ with $\|f\|_{\As_{pq}^{s\tau}(U)}:=\inf_{\tilde f|_U=f}\|\tilde f\|_{\As_{pq}^{s\tau}(\R^n)}$.

For the classical \textbf{Besov} and \textbf{Triebel-Lizorkin} spaces, we use 
\begin{align*}
    \Bs_{pq}^s(U):=\Bs_{p,q}^{s,0}(U)(=\Ns_{p,q}^{s,0}(U)),\quad\Fs_{pq}^s(U):=\Fs_{p,q}^{s,0}(U),&\quad \text{for }p,q\in(0,\infty]\text{ and }s\in\R\quad (p<\infty\text{ for }\Fs\text{-cases});\\
    \Fs_{\infty q}^s(U):=\Fs_{q,q}^{s,\frac1q}(U),&\quad\text{for }q\in(0,\infty]\text{ and }s\in\R.
\end{align*}
\end{defn}

We shall see that  $\As_{pq}^{s\tau}(U)$ is always a (quasi-)Banach space, and different choices of $(\lambda_j)_{j=0}^\infty$ result in equivalent norms. See \cite[Proposition~2.3.2]{TriebelTheoryOfFunctionSpacesI} and \cite[Propositions 1.3 and 1.8]{TriebelTheoryOfFunctionSpacesIV}.
\begin{rem}\label{Rmk::RmkSpace}
In the definition we only consider $\tau\le\frac1p$. If we extend them to $\tau>\frac1p$ then by \cite[Theorem~2]{YangYuanMorreyEquiv} and \cite[Lemma~3.4]{SickelMorrey1} we have $\Bs_{pq}^{s\tau}(\R^n)=\Fs_{pq}^{s\tau}(\R^n)=\Bs_{\infty\infty}^{s+n(\tau-\frac1p)}$, for all $0<p,q\le\infty$ and $s\in\R$.

When $\tau=\frac1p$, by \cite[Theorem~2]{YangYuanMorreyEquiv} and \cite[Remark~11(ii)]{SickelMorrey1} we have
\begin{equation}\label{Eqn::EqvSpaces2}
    \Bs_{p,\infty}^{s,\frac1p}(\R^n)=\Fs_{p,\infty}^{s,\frac1p}(\R^n)=\Bs_{\infty,\infty}^{s}(\R^n),\quad\Ns_{p,q}^{s,\frac1p}(\R^n)=\Bs_{\infty,q}^s(\R^n),\quad\forall\ 0<p,q\le\infty,\ s\in\R.
\end{equation} See also \cite[Proposition~1.18]{TriebelTheoryOfFunctionSpacesIV}.

Our notation $\Ns_{pq}^{s\tau}$ corresponds to the $\mathcal B_{pq}^{s\tau}$ in \cite[Definition~5]{SickelMorrey1}. For the usual conventions of Besov-Morrey spaces and Triebel-Lizorkin-Morrey spaces: we have $\mathcal N_{u,q,p}^s(\R^n)=\Ns_{p,q}^{s,\frac1p-\frac1u}(\R^n)$ and $\mathcal E_{u,q,p}^s(\R^n)=\Fs_{p,q}^{s,\frac1p-\frac1u}(\R^n)$ see \cite[Remark~13(iii) and Proposition~3.3 (iii)]{SickelMorrey1}.

For the notations in \cite{TriebelHybrid,HaroskeTriebelMorrey}, we have the correspondence $L^r\As_{pq}^s(\R^n)=\Lambda^{pr}\As_{pq}^s(\R^n)=\As_{p,q}^{s,\frac1p+\frac rn}(\R^n)$ for $\As\in\{\Bs,\Fs\}$ and $\Lambda_{pr}\Bs_{pq}^s(\R^n)=\Ns_{p,q}^{s,\frac1p+\frac rn}(\R^n)$, for $0<p,q\le\infty$, $s\in\R$ and $-\frac np\le r\le0$ ($p<\infty$ if $\As=\Fs$). See also \cite[Section 2.1]{HaroskeTriebelMorrey}.

For more discussions, we refer \cite{YSYMorrey,SickelMorrey1,TriebelHybrid,HaroskeTriebelMorrey} to readers.
\end{rem}

In the following, we will use the notation $x \lesssim y$ to mean $x \leq Cy$ where $C$ is a constant independent of $x,y$, and $x \approx y$ for ``$x \lesssim y$ and $y \lesssim x$''. We use $x\lesssim_\eps y$ to emphasize the dependence of $C$ on the parameter $\eps$.

For $r\neq0$ we use $\vartheta^r$ as the dilation operator given by 
\begin{equation}\label{Eqn::OpDila}
    \vartheta^rf(x',x_n):=f(x',rx_n)
\end{equation}

We use $M_{pq}^{ns}$ and $\widehat M_{pq}^{ns}$ for some fixed positive constants in Proposition \ref{Prop::Dila} and Theorem \ref{Thm::QntBFBdd}, both of which depend on $n\ge1$, $0<p,q\le\infty$ and $s\in\R$. 

We let $S=S_n$ be the zero extension operator of $\R^n_+$, that is
\begin{equation}\label{Eqn::OpS}
    Sf(x)=S_nf(x):=\1_{\R^n_+}(x)\cdot f(x)=\begin{cases}f(x',x_n)&x_n>0\\0&x_n<0\end{cases}.
\end{equation}
\begin{rem}\label{Rmk::OpSBdd}
    We have the definedness and boundedness $S_n:\As_{pq}^s(\R^n_+)\to\As_{pq}^s(\R^n)$ for $\max(n(\frac1p-1),\frac1p-1)<s<\frac1p$, see \cite[Theorem~2.48]{TriebelTheoryOfFunctionSpacesIV}. We will use the case $n=1$ and the case $p=\infty$ in the proof of Theorem \ref{MainThm} \ref{Item::Thm::BFBdd}.

    To clarify, the statement of \cite[Theorem~2.48]{TriebelTheoryOfFunctionSpacesIV} requires $q\ge1$ for the $\Fs_{\infty q}^s$-case. We do not need this restriction as we do not use the unboundedness of $S_n$ for the case $s\notin(\max(n(\frac1p-1),\frac1p-1),\frac1p)$. See also \cite[Theorem~2.48, Proof~Step~2]{TriebelTheoryOfFunctionSpacesIV}.
\end{rem}

\section{The Construction of Universal Seeley's Extension: Proof of Theorem \ref{MainThm} \ref{Item::Thm::Exist} and \ref{Item::Thm::SobBdd}}
First we need to fulfill condition \eqref{Eqn::CondAB}.  In \cite{Seeley} Seeley took $b_j:=2^j$ for $j\ge0$ and found $(a_j)_{j=0}^\infty$ such that $\sum_{j=0}^\infty a_j(-b_j)^k=1$ for all $k\ge0$. This can be done by constructing an entire function $F(z)=\sum_{j=0}^\infty a_jz^j$ such that $F(2^k)=(-1)^k$, we recall the construction below. 

\begin{lem}\label{Lem::1CV}
\begin{enumerate}[(i)]
    \item\label{Item::1CV::W} For $\beta>1$, $W_\beta(z):=\prod_{j=0}^\infty(1-z/\beta^j)$ ($z\in\C$) defines a nonzero entire function $W_\beta$, all of whose zeros are simple.
    \item\label{Item::1CV::F} Let $u=(u_k)_{k=0}^\infty\in\ell^\infty$. The function $F_u^\beta$ below is entire and satisfies $F_u^\beta(\beta^k)=u_k$, for all $k\ge0$:
\begin{equation}\label{Eqn::1CV::DefF}
    F_u^\beta(z):=\sum_{k=0}^\infty\frac{u_k}{W_\beta'(\beta^k)}\cdot\frac{W_\beta(z)}{z-\beta^k}.
\end{equation}
    \item\label{Item::1CV::Sum} When $\beta>2$, we have the following quantitative estimate for $l=1,2,3,\dots$:
    \begin{equation}\label{Eqn::1CV::Sum}
        \sum_{k=1}^\infty\Big|\frac{W_\beta(\beta^{-l})}{W_\beta'(\beta^k)\cdot(\beta^{-l}-\beta^k)}\Big|
        \le\frac{\beta^3\exp\frac2{\beta-1}}{(\beta^2-1)(\beta^2-\beta-1)}.
    \end{equation}
\end{enumerate}
\end{lem}
\begin{proof}\ref{Item::1CV::W} is the result of the Weierstrass factorization theorem, see \cite[Theorem~15.9]{BigRudin} for example.

\smallskip
For \ref{Item::1CV::F}, one can see that for $k\ge 0$,
\begin{equation}\label{Eqn::1CVPf::Wbeta'1}
    W_\beta'(\beta^k)=\frac{W_\beta(z)}{z-\beta^k}\Big|_{z=\beta^k}=-\beta^{-k}\prod_{j\ge0;j\neq k}(1-\beta^k/\beta^j)=-\beta^{-k}W_\beta(\beta^{-1})\cdot\prod_{l=1}^{k}(1-\beta^l).
\end{equation}
Since $\beta^l-1\ge\beta^{l-1}(\beta-1)$ for $l\ge1$, we see that for $k\ge0$:
\begin{equation}\label{Eqn::1CVPf::Wbeta'2}
    |W_\beta'(\beta^k)|\ge|W_\beta(\beta^{-1})|\beta^{-k}\prod_{l=1}^{k}\beta^{l-1}(\beta-1)=|W_\beta(\beta^{-1})|(\beta-1)^{k}\beta^{-k}\cdot\beta^{\frac{k^2-k}2}.
\end{equation}

Since $\beta>1$, we see that the sum $\sum_{k=0}^\infty {u_k}/{W'_\beta(\beta^k)}$ absolutely converges. Therefore the sum in \eqref{Eqn::1CV::DefF} converges absolutely and locally uniformly in $z$. We conclude that $F_u^\beta$ is indeed an entire function.

For each $j\ge0$, we have $\frac{W_\beta(z)}{z-\beta^k}\big|_{z=\beta^j}=W'_\beta(\beta^k)$ when $j=k$ and $\frac{W_\beta(z)}{z-\beta^k}\Big|_{z=\beta^j}=0$ when $j\neq k$. Therefore $F_u^\beta(\beta^k)=u_k$, as for each $z=\beta^k$ there is only one nonzero term in the sum \ref{Item::1CV::F}.

\smallskip
For \ref{Item::1CV::Sum}, note that for each $l\ge1$, $W_\beta(\beta^{-l})=|W_\beta(\beta^{-l})|<1$ because $|1-\beta^{-l}/\beta^j|<1$ for each $j\ge 0$. Therefore using \eqref{Eqn::1CVPf::Wbeta'1} and \eqref{Eqn::1CVPf::Wbeta'2} we see that for $k,l\ge1$,
\begin{align*}
    \Big|\frac{W_\beta(\beta^{-l})}{W_\beta'(\beta^k)(\beta^{-l}-\beta^k)}\Big|\le\frac{W_\beta(\beta^{-1})^{-1}\cdot W_\beta(\beta^{-l})\cdot\beta^k}{(\beta-1)^{k}\beta^{\frac{k-1}2k}\cdot|\beta^{-l}-\beta^k|}\le\frac{W_\beta(\beta^{-1})^{-1}}{(\beta-1)^k\beta^{k-1}}\cdot\frac1{1-\beta^{-k-l}}
    \le \frac{W_\beta(\beta^{-1})^{-1}}{(\beta^2-\beta)^k}\cdot\frac\beta{1-\beta^{-2}}.
\end{align*}

On the other hand
\begin{align*}
    \big(W_\beta(\beta^{-1})\big)^{-1}& =\prod_{j=1}^\infty (1-\beta^{-j})^{-1}=\exp\sum_{j=1}^\infty\log\frac1{1-\beta^{-j}}
    \\
    &<\exp\sum_{j=1}^\infty\log(1+2\beta^{-j})&(\tfrac1{1-x}<1+2x\text{ when }0<x<\tfrac12)
    \\
    &<\exp\sum_{j=1}^\infty2\beta^{-j}=\exp\frac2{\beta-1}&(\log(1+x)<x\text{ when }0<x<1).
\end{align*}
Therefore taking sum over $k\ge1$ we obtain \eqref{Eqn::1CV::Sum}:
\begin{equation*}
    \sum_{k=1}^\infty\left|\frac{W_\beta(\beta^{-l})}{W_\beta'(\beta^k)(\beta^{-l}-\beta^k)}\right|
    \le \frac\beta{1-\beta^{-2}}\sum_{k=1}^\infty\frac{\exp\frac2{\beta-1}}{(\beta^2-\beta)^{k}}=\frac{\beta^3\exp\frac2{\beta-1}}{(\beta^2-1)(\beta^2-\beta-1)}.\qedhere
\end{equation*}
\end{proof}
\begin{rem}\label{Rmk::SeeleyCont1}
If one takes $u_k:=(-1)^k$ in Lemma~\ref{Lem::1CV} \ref{Item::1CV::F}, then the corresponding sequence $a_j:=\frac1{j!}(F_u^\beta)^{(j)}(0)$ ($j=0,1,2,\dots$) satisfies $\sum_j a_j\beta^{jk}=(-1)^k$ for all $k\ge0$. This is enough to prove the boundedness results in \cite{Seeley} ($C^k$ for $k\ge0$), but in our case we also need the equality to be true for $k<0$. However, one cannot construct an entire function $F$ such that $F(2^{-k})=(-1)^{-k}$ for all $k\ge0$.
\end{rem}
\begin{rem}\label{Rmk::SeeleyCont2}
Alternatively, one can construct the $(a_j)_j$ using Vandermonde determinant, which is the method of \cite{Seeley}. Let $M_1,M_2\in\N_0$ and let $b_{-M_1},\dots,b_{M_2}>0$ be distinct numbers. By Cramer's rule we see that the solutions of the equation system $\sum_{j=-M_1}^{M_2}a_jb_j^k=(-1)^k$ for $-M_1\le k\le M_2$ are (uniquely) determined by the following (see \cite[Page 149]{Adams} for example):
\begin{equation}\label{Eqn::FiniteA}
    a_j=(-b_j)^{M_1}\frac{V(b_{-M_1},\dots,b_{j-1},-1,b_{j+1},\dots,b_{M_2})}{V(b_{-M_1},\dots,b_{M_2})}=(-b_j)^{M_1}\hspace{-0.1in}\prod_{\substack{-M_1\le k\le M_2\\k\neq j}}\hspace{-0.05in}\frac{b_k+1}{b_k-b_j},\ -M_1\le j\le M_2.
\end{equation}
Here $V(x_{-M_1},\dots,x_{M_2})=\prod_{-M_1\le u<v\le M_2}(x_v-x_u)$.

Now we set $b_j:=\beta^j$ for some $\beta>1$. If we let $M_1=0$ and let $M_2\to+\infty$, we see that \eqref{Eqn::FiniteA} converges and produces the same sequence $(a_j)_{j=0}^\infty$ to Remark~\ref{Rmk::SeeleyCont1}. 
\end{rem}

We find $(a_j)_{j=-\infty}^\infty$ through the following:
\begin{prop}\label{Prop::ConsAB}
There are real numbers $(\tilde a_j)_{j=0}^\infty$ such that $\sum_{j=0}^\infty\tilde a_j(4^{jk}+4^{-jk})=(-1)^k$ for all $k=0,1,2,\dots$, with all the summations converging absolutely.

Therefore the sequences $a_j:=\tilde a_{|j|}$ ($j\neq0$) $a_0:=2\tilde a_0$ and $b_j:=4^j$ ($j\in\Z$) fulfill condition \eqref{Eqn::CondAB}.

\end{prop}The idea is to construct an iterated sequence $\tilde a^\nu=(\tilde a_0^\nu,\tilde a_1^\nu,\tilde a_2^\nu,\dots)$ such that
\begin{equation*}
    \sum_{j=0}^\infty\tilde a_j^\nu4^{jk}=(-1)^k-\sum_{j=0}^\infty\tilde a_j^{\nu-1}4^{-jk},\quad\forall\, k\ge0.
\end{equation*}
The limit $\lim_{\nu\to\infty}\tilde a_j^\nu$ will be our $\tilde a_j$, provided that it is convergent. In practice, it is more convenient to prove the convergence of $(u^\nu)_{\nu=0}^\infty\subset\ell^\infty$ via the correspondence $u^\nu_k:=\sum_{j=0}^\infty \tilde a_j^\nu 4^{jk}$ and $\tilde a_j^\nu=\frac1{j!}\frac {d^j}{dz^j}F_{u^\nu}^4(z)|_{z=0}$, where $F_{u^\nu}^4(z)$ is the entire function in Lemma~\ref{Lem::1CV} \ref{Item::1CV::F} with $\beta=4$ and $u=u^\nu$.
\begin{proof}
Let $u^0=(u^0_k)_{k=0}^\infty\in\ell^\infty$ be $u^0_0:=\frac12$ and $u^0_k:=(-1)^k$ for $k\ge1$. For each $\nu\ge1$, we let $u^\nu_0:=\frac12$ and
\begin{equation*}
    u^\nu_k:=(-1)^k-\sum_{j=0}^\infty\Big(\tfrac1{j!}\frac {d^j}{dz^j}F_{u^{\nu-1}}^4(z)\Big|_{z=0}\Big)4^{-jk}=(-1)^k-F_{u^{\nu-1}}^4(4^{-k}),\quad\text{for } k\ge1.
\end{equation*}

By Lemma~\ref{Lem::1CV} \ref{Item::1CV::F}, if $u^{\nu-1}\in\ell^\infty$ then $F_{u^{\nu-1}}(z)$ is entire in $|z|<1$. Since $u^0\in\ell^\infty$, by recursion we get $\|u^\nu\|_{\ell^\infty}\le1+\|F_{u^{\nu-1}}\|_{L^\infty[0,1]}<\infty$ for every $\nu\ge1$. Thus we have a sequence of bounded sequences $(u^\nu)_{\nu=0}^\infty\subset\ell^\infty$.
% Liding: There is actually a problem here, even though we have $\|u^{\nu+1}-u^\nu\|_{\ell^\infty}<c\|u^\nu-u^{\nu-1}\|_{\ell^\infty}$, we DON'T know yet $\|u^1-u^0\|_{\ell^\infty}<\infty$, i.e. $u^1\in\ell^\infty$! So I add a \eqref{Eqn::1CV::SumCor}.}\par
% HL: But we can directly use the boundness of an entire function in unit disk, a compact set. We don't need universal bound. So proof can be shorter.}

% By \eqref{Eqn::1CV::SumCor} we have $\|u^\nu\|_{\ell^\infty}\le 1+C_\beta\|u^{\nu-1}\|_{\ell^\infty}\le (1+C_\beta)^\nu<\infty$ for all $\nu\ge1$, which means $u^\nu\in\ell^\infty$.

Moreover since $u^\nu_0\equiv\frac12$, by \eqref{Eqn::1CV::Sum},
\begin{align*}
    \|u^{\nu+1}-u^{\nu}\|_{\ell^\infty}&=\sup_{l\ge1}|u^{\nu+1}_l-u^\nu_l|\le \sup_{l\ge1}|F_{u^\nu}^4(4^{-l})-F_{u^{\nu-1}}^4(4^{-l})|
    \\
    &\le\sup_{l\ge1}\sum_{k=1}^\infty\Big|\frac{(u^\nu_k-u^{\nu-1}_k)\cdot W_4(4^{-l})}{W_4'(4^k)\cdot(4^{-l}-4^k)}\Big|\le \|u^\nu-u^{\nu-1}\|_{\ell^\infty}\frac{4^3\exp\frac2{4-1}}{(4^2-1)(4^2-4-1)}=\frac{64 e^{\frac23}}{165} \cdot\|u^\nu-u^{\nu-1}\|_{\ell^\infty}.
\end{align*}
Here $\frac{64 e^{2/3}}{165}\approx0.755<1$. We conclude that $(u^\nu)_{\nu=0}^\infty\subset\ell^\infty$ is a contraction sequence and must have a limit $u^\infty\in\ell^\infty$. Clearly $u^\infty_0=\frac12$ and $u^\infty_k=(-1)^k-F_{u^\infty}^4(4^{-k})$ for $k\ge1$.

By Lemma~\ref{Lem::1CV} \ref{Item::1CV::F}, the function $F^4_{u^\infty}(z)$ is entire. We define $\tilde a_j$ in the way that $F_{u^\infty}^4(z)=\sum_{j=0}^\infty\tilde a_jz^j$, thus
\begin{align*}
    \sum_{j=0}^\infty\tilde a_j&=u^\infty_0=1-u^\infty_0=(-1)^0-\sum_{j=0}^\infty\tilde a_j,& k=0,
    \\
    \sum_{j=0}^\infty\tilde a_j4^{jk}&=u^\infty_k=(-1)^k-F_{u^\infty}^4(4^{-k})=(-1)^k-\sum_{j=0}^\infty\tilde a_j4^{-jk},&k\ge1.
\end{align*}
Since $F^4_{u^\infty}$ is entire, the sum $\sum_{j=0}^\infty\tilde a_j4^{jk}(=F^4_{u^\infty}(4^k))$ always converge absolutely. Therefore $(\tilde a_j)_{j=0}^\infty$ is as desired.

Let $a_j:=\tilde a_{|j|}$ for $j\neq0$ and $a_0:=2\tilde a_0$, we see that $\sum_{j=-\infty}^\infty a_j4^{jk}=(-1)^k$ for all $k\in\Z$ with all summations converging absolutely. In particular $\sum_{j=-\infty}^\infty|a_j|4^{|j|k}<\infty$ for all $k\ge1$, which means that $\sum_{j=-\infty}^\infty|a_j|2^{\delta|j|}4^{jk}<\infty$ for all $\delta>0$ and $k\in\Z$. Therefore $(a_j,4^j)_{j=-\infty}^\infty$ satisfies the condition \eqref{Eqn::CondAB} and we prove Theorem \ref{MainThm} \ref{Item::Thm::Exist}.
\end{proof}

%We start the proof of Theorem \ref{MainThm} \ref{Item::Thm::SobBdd}. 

One can see that Seeley-type extensions have the following structures: % The direct computations yield the following. We omit the proof to reader.
\begin{lem}\label{Lem::Trivial}
\begin{enumerate}[(i)]
    \item\label{Item::Trivial::ab} Let $(a,b)=(a_j,b_j)_{j=-\infty}^\infty$ be the collection satisfying the condition \eqref{Eqn::CondAB}. Then for every $k\in\Z$ the collections $(a(-b)^k,b):=(a_j(-b_j)^k,b_j)_{j=-\infty}^\infty$ and $(a(-b)^k,\frac1b):=(a_{j}(-b_{j})^k,\frac1{b_{j}})_{j=-\infty}^\infty$ also satisfy \eqref{Eqn::CondAB}.
    \item\label{Item::Trivial::Comm} For multi-index $\gamma=(\gamma',\gamma_n)\in\N_0^n$ we have $\partial^\gamma\circ E^{a,b}=E^{a(-b)^{\gamma_n},b}\circ\partial^\gamma$ and $E^{a,b}\circ\partial^\gamma=\partial^\gamma\circ E^{a(-b)^{-\gamma_n},b}$.
    \item\label{Item::Trivial::E*} Let $E=E^{a,b}$ be the extension operator defined in \eqref{Eqn::ExtOp}. The formal adjoint $E^*$ has the following expression:
\begin{equation}\label{Eqn::EAdj}
E^*g(x',x_n)=g(x',x_n)+\sum_{j=-\infty}^\infty \frac{a_{j}}{b_{j}}g\big(x',-\tfrac1{b_{j}}\cdot x_n\big),\quad x'\in\R^{n-1},\quad x_n>0.
\end{equation}
\end{enumerate}
\end{lem}
\begin{proof}
    The results follow from direct computations.

    \ref{Item::Trivial::ab}: By assumption, $\sum_j a_j(-b_j)^k(-b_j)^l=\sum_ja_j(-b_j)^{k+l}=1$ and $\sum_j a_j(-b_j)^k(-\frac1{b_j})^l=\sum_ja_j(-b_j)^{k-l}=1$, both hold for every $l\in\Z$. Thus $(a(-b)^k,b)$ and $(a(-b)^k,\frac1b)$ both satisfy \eqref{Eqn::CondAB}.

    \medskip
    \ref{Item::Trivial::Comm}: Clearly $\partial_{x_j}\circ E=E\circ\partial_{x_j}$ for every $j\neq n$. By the chain rule $\partial_{x_n}^k(f(x',-b_jx_n))=(-b_j)^k\cdot(\partial_{x_n}^kf)(x',-b_jx_n)$ holds for every $k\ge0$. Therefore $\partial^\gamma(f(x',-b_jx_n))=(-b_j)^{\gamma_n}\cdot(\partial^\gamma f)(x',-b_jx_n)$. Taking sum over $j$ we get $\partial^\gamma\circ E^{a,b}=E^{a(-b)^{\gamma_n},b}\circ\partial^\gamma$. Replacing $(a,b)$ by $(a(-b)^{-\gamma_n},b)$ we see that $\partial^\gamma\circ E^{a(-b)^{-\gamma_n},b}=E^{a(-b)^{-\gamma_n}(-b)^{\gamma_n},b}\circ\partial^\gamma=E^{a,b}\circ\partial^\gamma$.

    \medskip
    \ref{Item::Trivial::E*}: For every $f\in C_c^\infty(\R^n_+)$ and $g\in C_c^\infty(\R^n)$,
    \begin{align*}
        \int_{\R^n} Ef(x)g(x)dx=&\int_{\R^n_+}f(x)g(x)+\sum_{j=-\infty}^\infty\int_{-\infty}^0 \int_{\R^{n-1}}a_jf(x',-b_jx_n)g(x',x_n)dx'dx_n
        \\
        =&\int_{\R^n_+}f(x)g(x)dx+\sum_{j=-\infty}^\infty\int_{-\infty}^0 \int_{\R^{n-1}}a_jf(x',y_n)g\Big(x',-\frac1{b_j}y_n\Big)dx'\,\frac{dy_n}{-b_j}
        \\
        =&\int_{\R^n_+}f(x)g(x)dx+\sum_{j=-\infty}^\infty\int_0^\infty \int_{\R^{n-1}}\frac{a_j}{b_j}f(x',y_n)g\big(x',-\tfrac1{b_j}y_n\big)dx'dy_n.
    \end{align*}
    Replacing $y_n$ by $x_n$ we get the expression \eqref{Eqn::EAdj}.
\end{proof}

% In the proof later we sometimes take zero extension of $E^*g$ to $x_n<0$, while sometimes extend \eqref{Eqn::EAdj} to $x_n<0$. Based on the proof, one can see that these two different extensions are both bounded in all the spaces we want.

\begin{rem}\label{Rmk::RmkofTrivialComm}
In the notations of \eqref{Eqn::OpDila} and \eqref{Eqn::OpS} we can write $E^{a,b}=S+\sum_ja_j\vartheta^{-b_j}\circ S$. For $r<0$ we have $$\partial_{x_n}(\vartheta^rSf)(x)=r(S\partial_{x_n}f)(x',rx_n)-\delta_0(x_n)f(x',rx_n),$$ where $\delta_0$ is the Dirac measure at $0\in\R$. 

Therefore $\partial_n\circ E^{a,b}=E^{a(-b),b}\circ\partial_n$ (in the domain where either side is defined) if and only if $\sum_ja_j=1$.
Taking higher order derivatives we see that $\partial^\gamma\circ E^{a,b}=E^{a(-b)^{\gamma_n},b}\circ\partial^\gamma$ still holds if we only assume the condition $\sum_ja_j(-b_j)^k=1$ for $0\le k\le \gamma_n$.
\end{rem}

\begin{lem}\label{Lem::TriSum}
Let $(X,\|\cdot\|)$ be a quasi-Banach space such that $(x,y)\mapsto\|x-y\|^q$ is a metric for some $0<q\le1$, i.e. $\|x+y\|^q\le \|x\|^q+\|y\|^q$ for all $x,y\in X$. Then for any $\delta>0$ there is a constant $K_{q,\delta}>0$ such that 
$\big\|\sum_{j=-\infty}^\infty x_j\big\|\le K_{q,\delta}\sum_{j=-\infty}^\infty2^{\delta|j|}\|x_j\|$
for all sequence $(x_j)_{j\in\Z}\subset X$ such that right hand summation converges. In this case the sum $\sum_{j=-\infty}^\infty x_j$ converges with respect to the quasi-norm topology of $X$.
\end{lem}
\begin{rem}
    The couple $(X,\|\cdot\|^q)$ is also call a \textit{$q$-convex quasi Banach space} ($q$-Banach space for short). 
    
    In practice we use $X=W^{k,q}(U)$ and $X=\As_{pr}^{s\tau}(U)$ (see Definition \ref{Defn::BsFs}) with $q=\min(p,r,1)$.
\end{rem}
% Note that such $q>0$ always exists. Indeed, let $C\ge1$ be such that $\|x+y\|\le C(\|x\|+\|y\|)$, and let $k\in\Z_+$ be such that $2^{k-1}\ge C$, using binomial expansion we can see that $C(\|x\|+\|y\|)\le(\|x\|^{1/k}+\|y\|^{1/k})^k$, thus $q:=1/k$ is as desired. (Liding: Ignore this this is not true)
\begin{proof}[Proof of Lemma \ref{Lem::TriSum}]
When $(x_j)_{j\in\Z}$ are finitely nonzero, by H\"older's inequality we have 
$$\Big\| {\sum_j x_j} \Big\|_X\le \Big|\sum_j\|x_j\|_X^q\Big|^{1/q}\le\Big(\big\|(2^{\delta q|j|}\|x_j\|_X^q)_{j\in\Z}\big\|_{\ell^{\frac1q}(\Z)}\big\|(2^{-\delta q|j|})_{j\in\Z} \big\|_{\ell^{\frac1{1-q}}(\Z)}\Big)^{1/q}=K_{q,\delta} \sum_j 2^{\delta |j|}\|{x_j} \|_X.  $$
Here $K_{q,\delta}:=(\frac{2^{\delta/(1-q)}+1}{2^{\delta/(1-q)}-1})^{1/q-1}$ if $q<1$ and $K_{q,\delta}:=1$ if $q=1$.

When $\sum_{j=-\infty}^\infty2^{\delta|j|}\|x_j\|<\infty$, we have  $\lim\limits_{\nu\to+\infty}\sup\limits_{\nu\le\nu'<\nu''}\big\|\sum\limits_{\nu'<|j|\le\nu''}x_j\big\|\le K_{q,\delta}\lim\limits_{\nu\to+\infty}\sum\limits_{|j|\ge\nu}2^{\delta|j|}\|x_j\|=0$. Therefore $\big(\sum_{|j|\le\nu}x_j\big)_{\nu=1}^\infty$ is a Cauchy sequence in $X$. Since $X$ is complete, the sum must converge to a unique limit element and the same inequality $\|\sum_jx_j\|\le K_{q,\delta}\sum_j2^{\delta|j|}\|x_j\|$ holds. 
\end{proof}

The proof of boundedness on Sobolev spaces is standard:
\begin{proof}[Proof of Theorem~\ref{MainThm} \ref{Item::Thm::SobBdd}] We only need to prove the Sobolev boundedness $W^{k,p}$ for $k\ge0$, $0<p\le\infty$. 

Indeed suppose we get $E:W^{k,\infty}(\R^n_+)\to W^{k,\infty}(\R^n)$, then $C^\infty(U)=\bigcap_{k=0}^\infty W^{k,\infty}(U)$ for $U\in\{\R^n,\R^n_+\}$ gives $E:C^\infty(\R^n_+)\to C^\infty(\R^n)$. Clearly $C^k(U)=\overline{C^\infty(U)}^{\|\cdot\|_{W^{k,\infty}(U)}}$ for $U\in\{\R^n,\R^n_+\}$, thus $E:C^k(\R^n_+)\to C^k(\R^n)$ holds for all $k\ge0$.

The $C^{k,s}$ boundedness ($k\in\Z$, $0<s<1$) can follow either from the same argument for $W^{k,p}$ below, or the Besov correspondence $C^{k,s}(U)=\Bs_{\infty\infty}^{k+s}(U)$, for $U\in\{\R^n,\R^n_+\}$ (see Proposition~\ref{Prop::Stuff} \ref{Item::Stuff::Holder}).
When $k\ge0$ one can also obtain the boundedness by the interpolation $C^{k,s}(U)=(C^k(U),C^{k+1}(U))_{s,\infty}$ (see for example \cite[Theorem~2.7.2/1]{TriebelInterpolation} and the proof of \cite[Theorem~4.5.2/1]{TriebelInterpolation}).

To prove the $W^{k,p}$-boundedness we apply Lemma~\ref{Lem::TriSum}. For every $f\in L^p(\R^n_+)$ if $1\le p\le\infty$, and every $f\in C_c^\infty(\overline{\R^n_+})$ if $0<p<1$, we have (recall the notation $\vartheta^r$ in \eqref{Eqn::OpDila})
\begin{equation}\label{Eqn::PfSobBdd::Comp}
    \begin{aligned}
    \|E^{a,b}f\|_{L^p(\R^n)}\le&\textstyle2^{\max(0,\frac1p-1)}\big(\| f\|_{L^p(\R^n_+)}+\big\|\sum_{j\in\Z}a_j
\cdot\vartheta^{-b_j}f\big\|_{L^p(\R^n_-)}\big)
    \\
    \lesssim&_{p,\delta}\| f\|_{L^p(\R^n_+)}+\sum_{j\in\Z}2^{\delta|j|}|a_j|\Big(\int_{\R^n_-}|f(x',-b_jx_n)|^pdx'dx_n\Big)^\frac1p&\quad\text{(by Lemma~\ref{Lem::TriSum})}
    \\
    =&\|f\|_{L^p(\R^n_+)}+\sum_{j\in\Z}2^{\delta|j|}|a_j|\Big(b_j^{-1}\int_{\R^n_+}|f(x',y_n)|^pdx'dy_n\Big)^\frac1p\\
    =&\Big(1+\sum_{j\in\Z}2^{\delta|j|}|a_j|b_j^{-\frac1p}\Big)\|f\|_{L^p(\R^n_+)}.
\end{aligned}
\end{equation}

%要不要补一句 delta 是来自 lemma 10，从而我们后面才能 taking delta small enough 直到满足 Thm1 的 delta.
When $k\ge0$, also for every $f\in L^p(\R^n_+)$ if $1\le p\le\infty$, and every $f\in C_c^\infty(\overline{\R^n_+})$ if $0<p<1$, applying Lemma \ref{Lem::Trivial} \ref{Item::Trivial::Comm} to \eqref{Eqn::PfSobBdd::Comp} we have
\begin{align*}
    \|E^{a,b}f\|_{W^{k,p}(\R^n)}\lesssim&_{k,p}\sum_{|\gamma|\le k}\|\partial^\gamma E^{a,b} f\|_{L^p(\R^n)}=\sum_{|\gamma|\le k}\|E^{a(-b)^{\gamma_n},b}\partial^\gamma f\|_{L^p(\R^n)}
    \\
    \lesssim&_{p,\delta}\sum_{|\gamma|\le k}\sum_{j\in\Z}2^{\delta|j|}|a_j|b_j^{\gamma_n-\frac1p}\|\partial^\gamma f\|_{L^p(\R^n_+)}\lesssim_{k,p}\sum_{j\in\Z}2^{\delta|j|}|a_j|(b_j^{-\frac1p}+b_j^{k-\frac1p})\| f\|_{W^{k,p}(\R^n_+)}.
\end{align*}

By the assumption \eqref{Eqn::CondAB} we can take $\delta>0$ small such that $\sum_{j\in\Z}2^{\delta|j|}|a_j|b_j^{k-\lceil\frac1p\rceil}<\infty$ and $\sum_{j\in\Z}2^{\delta|j|}|a_j|b_j^{k}<\infty$. We now get the boundedness of $W^{k,p}$ for $k\ge0$.

\smallskip
When $k<0$, applying Lemma \ref{Lem::Trivial} \ref{Item::Trivial::Comm} to \eqref{Eqn::PfSobBdd::Comp} for
\begin{itemize}
    \item every $f\in W^{k,p}(\R^n_+)$ and $\{g_\gamma\}_{|\gamma|\le-k}\subset L^p(\R^n_+)$ such that $f=\sum_{|\gamma|\le -k}\partial^\gamma g_\gamma$, if $1\le p\le\infty$;
    \item every $f\in C_c^\infty(\overline{\R^n_+})$ and $\{g_\gamma\}_{|\gamma|\le-k}\subset  C_c^\infty(\overline{\R^n_+})$ such that $f=\sum_{|\gamma|\le -k}\partial^\gamma g_\gamma$, if $0<p<1$,
\end{itemize}
we have the following
\begin{align*}
    \|E^{a,b}f\|_{W^{k,p}(\R^n)}\lesssim&_{k,p}\sum_{|\gamma|\le -k}\| E^{a,b}\partial^\gamma g_\gamma\|_{W^{k,p}(\R^n)}=\sum_{|\gamma|\le -k}\|\partial^\gamma E^{a(-b)^{-\gamma_n},b} g_\gamma\|_{W^{k,p}(\R^n)}
    \\
    \le&\sum_{|\gamma|\le -k}\|E^{a(-b)^{-\gamma_n},b} g_\gamma\|_{L^p(\R^n)}\lesssim_{p,\delta}\sum_{|\gamma|\le -k}\sum_{j\in\Z}2^{\delta|j|}|a_j|b_j^{-\gamma_n-\frac1p}\|g_\gamma\|_{L^p(\R^n_+)}
    \\
    \le&\sum_{j\in\Z}2^{\delta|j|}|a_j|(b_j^{-\frac1p}+b_j^{k-\frac1p})\sum_{|\gamma|\le -k}\|g_\gamma\|_{L^p(\R^n_+)}.
\end{align*}

Taking the infimum over all $\{g_\gamma\}_{|\gamma|\le -k}$ for $f$, and taking $\delta$ from \eqref{Eqn::CondAB} such that $\sum_{j\in\Z}2^{\delta|j|}|a_j|(b_j^k+b_j^{k-\lceil\frac1p\rceil})<\infty$ as well, we get the $W^{k,p}$-boundedness for $k<0$, finishing the proof.
\end{proof}
\begin{rem}\label{Rmk::PfSobBdd::VectValuedTmp}
The result can be extended to vector-valued functions. Indeed let $0<q\le1$, $0<p<\infty$, and let $(X,|\cdot|_X)$ be a quasi Banach space such that $|x+y|_X^q\le|x|_X^q+|y|_X^q$. Then for every $L^p$ (strongly measurable) function $f:\R^n_+\to X$, by the same calculation to \eqref{Eqn::PfSobBdd::Comp},
\begin{equation*}%\label{Eqn::PfSobBdd::VectValuedTmp}
    \begin{aligned}
    &\bigg(\int_{\R^n_-}\Big|\sum_{j\in\Z}a_jf(x',-b_jx_n)\Big|_X^pdx'dx_n\bigg)^\frac1p\le K_{q,\frac\delta2}\bigg(\int_{\R^n_-}\Big(\sum_{j\in\Z}2^{|j|\frac\delta2}|a_j||f(x',-b_jx_n)|_X\Big)^pdx'dx_n\bigg)^\frac1p
    \\
    \le&K_{q,\frac\delta2}K_{p,\frac\delta2}\sum_{j\in\Z}2^{|j|(\frac\delta2+\frac\delta2)}|a_j|\bigg(\int_{\R^n_+}b_j^{-1}|f(x',y_n)|_X^pdx'dy_n\bigg)^\frac1p=\Big(K_{q,\frac\delta2}K_{p,\frac\delta2}\sum_{j\in\Z}2^{|j|\delta}|a_j|b_j^{-\frac1p}\Big)\|f\|_{L^p(\R^n_+;X)}.
\end{aligned}
\end{equation*}
Here $K_{q,\delta/2}$ and $K_{p,\delta/2}$ are the constants in Lemma~\ref{Lem::TriSum}. 
Also by Lemma~\ref{Lem::TriSum} the sum of functions $\sum_{j\in\Z}a_jf(x',-b_jx_n)$ converges in $L^p(\R^n_-;X)$.

Therefore $E^{a,b}$ defines a bounded linear map $E^{a,b}:L^p(\R^n_+;X)\to L^p(\R^n;X)$ with the control of operator norms $\|E^{a,b}\|_{L^p(\R^n_+;X)\to L^p(\R^n;X)}\lesssim_{p,q,\delta}\sum_{j\in\Z}2^{|j|\delta}|a_j|b_j^{-1/p}$.

In practice we will take $X=\Bs_{pq}^s(\R^{n-1})$ in the proof of Theorem \ref{Thm::QntBFBdd}, \hyperlink{HTag::PfQStep3}{Step A3}.
\end{rem}

\section{Boundedness on Besov and Triebel-Lizorkin: Proof of Theorem \ref{MainThm} \ref{Item::Thm::BFBdd} and \ref{Item::Thm::MorreyBdd}}

Before we discuss the boundedness on Besov and Triebel-Lizorkin spaces, for completeness we prove that $E:\Ss'\to\Ss'$ is defined. Recall Definition \ref{Defn::Ss} that $\Ss'(\R^n_+)=\Ss'(\R^n)/\{\tilde f:\tilde f|_{\R^n_+}=0\}$.

\begin{lem}
$E:\Ss'(\R^n_+)\to\Ss'(\R^n)$ is defined and bounded. 
\end{lem}
\begin{proof} Since $\Ss'(\R^n_+)=\Ss(\R^n_+)'$ (see Remark~\ref{Rmk::DefSs} \ref{Item::DefSs::DualS}), by duality it is equivalent to show that $E^*:\Ss(\R^n)\to\Ss(\R^n_+)$ is continuous.

% For all $f\in\Ss'(\R^n_+)$, there exist $C_f,m,k$ (depending on $f$) satisfying 
% \begin{equation*}
%     |f(g)|\le C_f\sum_{|\alpha|\le m;|\beta|\le k}\sup_{x\in\R^n_+}|x^\alpha\partial^\beta g(x)|\qquad\forall\, g\in\Ss(\R^n_+).
% \end{equation*}
% Hence we have
% \begin{equation*}
%     |Ef(g)|=|f(E^*g)|\le C_f\sum_{|\alpha|\le m;|\beta|\le k}\sup_{x\in\R^n_+}|x^\alpha\partial^\beta E^*g(x)|\qquad \forall\,g\in\Ss(\R^n).
% \end{equation*}
Let $g\in\Ss(\R^n)$, using \eqref{Eqn::EAdj} we see that for every $\alpha,\beta\in\N_0^n$ and $x\in\R^n_+$,
\begin{align*}
    |x^\alpha\partial^\beta E^*g(x)| \le |x^\alpha\partial^\beta g(x)|+\sum_{j\in\Z}\big|a_j(-b_j)^{\alpha_n-\beta_n-1}\big(x',-\tfrac{1}{b_j}\cdot x_n\big)^\alpha(\partial^\beta g) (x',-\tfrac{x_n}{b_j})\big|.
\end{align*}

By \eqref{Eqn::CondAB}, $\sum_{j\in\Z}|a_j|\cdot b_j^{\alpha_n-\beta_n-1}<\infty$. We conclude that $|x^\alpha\partial^\beta E^*g(x)|\lesssim_{a,b,\alpha,\beta}\sup_{y\in\R^n}|y^\alpha\partial^\beta g(y)|$ for all $x\in\R^n_+$.

On the other hand, by the condition $\sum_ja_j(-b_j)^{-k}=1$ for $k=1,2,3,\dots$ in \eqref{Eqn::CondAB}, we have
\begin{equation*}
    \lim_{x_n\to0+}\partial^\beta E^*g(x)=\partial^\beta g(x',0)-\sum_{j\in\Z}a_j(-b_j)^{-\beta_n-1}(\partial^\beta g)(x',0)=0,\quad\forall \beta\in\N_0^n.
\end{equation*}

Therefore along with the zero extension to $\R^n_-$, $E^*g$ defines a Schwartz function in $\R^n$ which has support $\overline{\R^n_+}$. Moreover $\sup_{x\in\R^n}|x^\alpha\partial^\beta E^*g(x)|\lesssim_{a,b,\alpha,\beta}\sup_{y\in\R^n}|y^\alpha\partial^\beta g(y)|$ with the implied constant being independent of $g$.

Since the seminorms $\big\{\sup_{x\in\R^n}|x^\alpha\partial^\beta E^*g(x)|\big\}_{\alpha,\beta\in\N_0^n}$ define the topology of $\Ss(\R^n)$, we conclude that\\ $E^*:\Ss(\R^n)\to\Ss(\R^n_+)$ is continuous. Therefore $E:\Ss'(\R^n_+)\to\Ss'(\R^n)$ is continuous as well.
\end{proof}
% The boundedness $E^*:\Ss'(\R^n)\to\Ss'(\R^n_+)$ follows by the similar but simpler calculations. We omit the proof to reader.

We now turn to prove Theorem~\ref{MainThm} \ref{Item::Thm::BFBdd} and \ref{Item::Thm::MorreyBdd}. 

To prove Theorem~\ref{MainThm} \ref{Item::Thm::BFBdd}, it might be possible to use the mean oscillation approach in \cite[Chapter 4.5.5]{TriebelTheoryOfFunctionSpacesII}. However \cite[(4.5.2/9)]{TriebelTheoryOfFunctionSpacesII} fails when $(a_j)_j$ has infinite nonzero terms. This is critical if one wants to prove the Besov and Triebel-Lizorkin boundedness of Seeley's operator (see Remark~\ref{Rmk::SeeleyCont1}) using such method. The authors do not know what the correct modification of \cite[(4.5.2/9)]{TriebelTheoryOfFunctionSpacesII} is.

Instead, we adapt the arguments from \cite[Section 2.9]{TriebelTheoryOfFunctionSpacesI}. The proof involves characteristic function multiplier \cite[Section 2.8.7]{TriebelTheoryOfFunctionSpacesI}, the Fubini decomposition \cite[Section 2.5.13]{TriebelTheoryOfFunctionSpacesI} and interpolations. We recall these facts from literature.
\begin{prop}\label{Prop::Stuff}
Let $0<p,q\le\infty$, $s\in\R$ and $0\le\tau\le\frac1p$. Let $U\subseteq\R^n$ be an arbitrary open subset.
\begin{enumerate}[(i)]

    \item\label{Item::Stuff::SumDer} For $\As\in\{\Bs,\Fs\}$ and for every integer $m\ge1$, $\As_{pq}^s(U)=\big\{\sum_{|\alpha|\le m}\partial^\alpha g_\alpha:g_\alpha\in\As_{pq}^{s+m}(U)\big\}$ and has equivalent (quasi-)norm $\|f\|_{\As_{pq}^s(U)}\approx_{p,q,s,m}\inf\big\{\sum\limits_{|\alpha|\le m}\|g_\alpha\|_{\As_{pq}^{s+m}(U)}:f=\sum\limits_{|\alpha|\le m}\partial^\alpha g_\alpha\text{ as distributions}\big\}$.
    
    \item\label{Item::Stuff::Holder}If $U\in\{\R^n,\R^n_+\}$ or $U$ is a bounded smooth domain, then $C^{k,s}(U)=\Bs_{\infty\infty}^s(U)$ for all $k\in\Z$ and $0<s<1$, with equivalent norms.
 
    \item{\normalfont(Fubini)}\label{Item::Stuff::Fubini} When $p<\infty$ or $p=q=\infty$, and when $s>n\max(0,\frac1p-1,\frac1q-1)$, we have decomposition $\Fs_{pq}^s(\R^n)=L^p(\R_{x_n};\Fs_{pq}^s(\R^{n-1}_{x'}))\cap L^p(\R^{n-1}_{x'};\Fs_{pq}^s(\R_{x_n}))$, with equivalent (quasi-)norms:
    \begin{equation}\label{Eqn::Stuff::Fubini}
        \|f\|_{\Fs_{pq}^s(\R^n)}^p\approx_{n,p,q,s}\int_\R\|f(\cdot,x_n)\|_{\Fs_{pq}^s(\R^{n-1})}^pdx_n+\int_{\R^{n-1}}\|f(x',\cdot)\|_{\Fs_{pq}^s(\R)}^pdx'.
    \end{equation}
    \item{\normalfont(Interpolations)} \label{Item::Stuff::Interpo} For $s_0,s_1\in\R$ and $\theta\in(0,1)$ such that $s_0\neq s_1$ and $s=(1-\theta)s_0+\theta s_1$, we have
    \begin{gather}
        \label{Eqn::Stuff::CpxInt}\quad\langle\Fs_{pq}^{s_0}(\R^n),\Fs_{pq}^{s_1}(\R^n),\theta\rangle=\Fs_{pq}^s(\R^n);
        \\
        \label{Eqn::Stuff::RealInt1}\quad (\Fs_{pp}^{s_0}(\R^n),\Fs_{pp}^{s_1}(\R^n))_{\theta,q}=\Bs_{pq}^{s}(\R^n);
        \\
        \label{Eqn::Stuff::RealInt2}\quad
             (\Fs_{pp}^{s_0,\tau}(\R^n),\Fs_{pp}^{s_1,\tau}(\R^n))_{\theta,q}=\Ns_{pq}^{s\tau}(\R^n),\quad \text{for }\tau\in[0,\tfrac1p)\quad(\text{provided }p<\infty).
    \end{gather}

    \item{\normalfont(Local representations)}\label{Item::Stuff::FInftyQ} For $s<0$, we have equivalent norms 
    \begin{align}
        \|f\|_{\Bs_{pq}^{s\tau}(\R^n)}\approx&_{n,p,q,s,\tau}\sup_{x\in\R^n;R>0}R^{-n\tau}\|f\|_{\Bs_{pq}^s(B(x,R))};\\
        \|f\|_{\Fs_{pq}^{s\tau}(\R^n)}\approx&_{n,p,q,s,\tau}\sup_{x\in\R^n;R>0}R^{-n\tau}\|f\|_{\Fs_{pq}^s(B(x,R))},\qquad\text{provided }p<\infty.
    \end{align}
    
\end{enumerate}
\end{prop}
\begin{proof}
For \ref{Item::Stuff::SumDer}, by \cite[Theorem~1.24]{TriebelTheoryOfFunctionSpacesIV} $\|\partial^\alpha \tilde g\|_{\As_{pq}^s(\R^n)}\lesssim\|\tilde g\|_{\As_{pq}^{s+m}(\R^n)}$ holds for $|\alpha|\le m$. Note that  $\tilde g|_U=g$ implies $\partial^\alpha\tilde g|_U=\partial^\alpha g$, thus taking restrictions we get $\|\sum_{|\alpha|\le m}\partial^\alpha g_\alpha\|_{\As_{pq}^s(U)}\lesssim\sum_{|\alpha|\le m}\|g_\alpha\|_{\As_{pq}^{s+m}(U)}$. Therefore $\|f\|_{\As_{pq}^s(U)}\lesssim\sum_{|\alpha|\le m}\|g_\alpha\|_{\As_{pq}^{s+m}(U)}$ whenever $f=\sum_{|\alpha|\le m}\partial^\alpha g_\alpha$.

Conversely for $f\in\As_{pq}^s(U)$, take  an extension $\tilde f\in\As_{pq}^s(\R^n)$, we have $\tilde f=(I-\Delta)(I-\Delta)^{-1}\tilde f=\tilde g_0+\sum_{j=1}^n\partial_j\tilde g_j$ where $\tilde g_0:=(I-\Delta)^{-1}\tilde f$ and $\tilde g_j:=-\partial_j(I-\Delta)^{-1}\tilde f$ ($1\le j\le n$). Thus by \cite[Theorems~1.22 and 1.24]{TriebelTheoryOfFunctionSpacesIV} we have $\sum_{j=0}^n\|\tilde g_j\|_{\As_{pq}^{s+1}(\R^n)}\lesssim\|\tilde f\|_{\As_{pq}^s(\R^n)}$. Taking infimum of all extensions $\tilde f$ over $f$ we get \ref{Item::Stuff::SumDer} for $m=1$. Doing this recursively we get all $m\ge1$.

\smallskip For \ref{Item::Stuff::Holder}, we have $C^{k,s}(\R^n)=\Bs_{\infty\infty}^{k+s}(\R^n)$ from \cite[Section~2.5.7]{TriebelTheoryOfFunctionSpacesI} for $k\ge0$ and $0<s<1$. Thus in this case $\Bs_{\infty\infty}^{k+s}(U)=\{\tilde f|_U:\tilde f\in C^{k,s}(\R^n)\}\subseteq C^{k,s}(U)$ for arbitrary open subset $U$. When $U=\R^n_+$ or $U$ is bounded smooth, by the existence of extension operator $C^{k,s}(U)\to C^{k,s}(\R^n)$ (see \cite[Lemma~6.37]{GilbargTrudinger}) we get $C^{k,s}(U)\subseteq\Bs_{\infty\infty}^{k+s}(U)$. This proves the case $k\ge0$.

Comparing the result \ref{Item::Stuff::SumDer} and \eqref{Eqn::NegHoldNorm} we see that $C^{k,s}(U)=\Bs_{\infty\infty}^{k+s}(U)$ for $k<0$ as well.

\smallskip
For \ref{Item::Stuff::Fubini} see \cite[Theorem~3.25]{TriebelTheoryOfFunctionSpacesIV} and \cite[Theorem~4.4]{TriebelStructureFunctions}  In the reference a further decomposition $\Fs_{pq}^s(\R^n)=\bigcap_{k=1}^n L^p(\R^{n-1}_{\widehat {x_k}};\Fs_{pq}^s(\R_{x_k}))$ is given, where $\widehat{x_k}=(x_1,\dots, x_{k-1}, x_{k+1},\dots, x_n)$. Our result is implied by the following:

Write $\widehat{x'_k}=(x_1,\dots, x_{k-1}, x_{k+1},\dots, x_{n-1})$ for $1\le k\le n-1$. By \cite[Theorem~3.25]{TriebelTheoryOfFunctionSpacesIV} we have
\begin{align*}
    \Fs_{pq}^s(\R^n_x)&=\bigcap_{k=1}^nL^p(\R^{n-1}_{\widehat {x'_k}};\Fs_{pq}^s(\R_{x_k})),&\|f\|_{\Fs_{pq}^s(\R^n)}^p\approx&_{n,p,q,s}\sum_{k=1}^n\int_{\R^{n-1}}\|f_{\widehat{x_k}}\|_{\Fs_{pq}^s(\R_{x_k})}^pd\widehat{x_k};
    \\
     \Fs_{pq}^s(\R^{n-1}_{x'})&=\bigcap_{k=1}^{n-1} L^p(\R^{n-2}_{\widehat {x'_k}};\Fs_{pq}^s(\R_{x_k})),&\|g\|_{\Fs_{pq}^s(\R^{n-1}_{x'})}^p\approx&_{n,p,q,s}\sum_{k=1}^{n-1}\int_{\R^{n-2}}\|g_{\widehat {x'_k}}\|_{\Fs_{pq}^s(\R_{x_k})}^pd\widehat{x'_k}.
\end{align*}
Here for each $\widehat{x_k}\in\R^{n-1}$, $f_{\widehat {x_k}}$ is the slide of $f$ with variable $x_k\in\R$. Similar notations for $g_{\widehat{x'_k}}$. Therefore we get \eqref{Eqn::Stuff::Fubini} by the following:
\begin{align*}
    \|f\|_{\Fs_{pq}^s(\R^n)}^p\approx&\textstyle\sum_{k=1}^{n-1}\int_{\R^{n-1}}\|f_{\widehat{x_k}}\|_{\Fs_{pq}^s(\R_{x_k})}^pd\widehat{x_k}+\int_{\R^{n-1}}\|f(x',\cdot)\|_{\Fs_{pq}^s(\R)}^pd\widehat{x_n}
    \\
    \approx&\textstyle\sum_{k=1}^{n-1}\int_{\R}\big(\int_{\R^{n-2}}\|f(\cdot,x_n)_{\widehat{x'
    _k}}\|_{\Fs_{pq}^s(\R_{x_k})}^pd\widehat{x'_k}\big)dx_n+\int_{\R^{n-1}}\|f(x',\cdot)\|_{\Fs_{pq}^s(\R)}^pdx'
    \\
    \approx&\textstyle\int_{\R}\|f(\cdot,x_n)\|_{\Fs_{pq}^s(\R^{n-1})}^pdx_n+\int_{\R^{n-1}}\|f(x',\cdot)\|_{\Fs_{pq}^s(\R)}^pdx'.
\end{align*}

\smallskip
For \ref{Item::Stuff::Interpo}, we refer the reader to \cite{YSY15}. 

Here \eqref{Eqn::Stuff::CpxInt} follows from \cite[Theorem~2.13 and Remark~2.13 (i), (iii)]{YSY15}. Here $\langle X_0,X_1,\theta\rangle$ is call the \textit{$\pm$-method}, which is different from the standard complex interpolation $[X_0,X_1]_\theta$. See \cite[Definitions 2.9 and 2.50]{YSY15} for example.

See \cite[Theorem~2.80]{YSY15} for \eqref{Eqn::Stuff::RealInt2} and \cite[Remark~2.81~(i)]{YSY15} for \eqref{Eqn::Stuff::RealInt1} when $p<\infty$. Recall from Remark~\ref{Rmk::RmkSpace} that we use $\Ns_{pq}^{s\tau}=\mathcal N_{\frac1p-\tau,p,q}^s$. 

The endpoint case $p=\infty$ of \eqref{Eqn::Stuff::RealInt1} follows from the standard real interpolation of Besov spaces. Indeed we have $\Fs_{\infty\infty}^s=\Bs_{\infty\infty}^s$ and $(\Bs_{\infty\infty}^{s_0}(\R^n),\Bs_{\infty\infty}^{s_1}(\R^n))_{\theta,q}=\Bs_{\infty q}^{s}(\R^n)$. See \cite[Theorem~2.4.2~(i)]{TriebelTheoryOfFunctionSpacesI} for example.

\smallskip
See \cite[Theorem~3.64]{TriebelHybrid} for \ref{Item::Stuff::FInftyQ}. In the reference the collection $\{B(x,2^{-J}):x\in\R^n,J\in\Z\}$ is replaced by $\{2Q_{J,v}:=2^{-J}(v+(-\frac12,\frac32)^n):J\in\Z,v\in\Z^n\}$, but there is no difference between the proofs. % Note that one cannot use $\{Q_{J,v}:=2^{-J}(v+(0,1)^n):J\in\Z,v\in\Z^n\}$ because for $J_0\in\Z$, the set $\{x\in\R^n:x_k\in2^{-J_0}\Z,\ \exists\,1\le k\le n\}$ is not contained in the union $\bigcup_{J\ge J_0;v\in\Z^n}Q_{J,v}$.
\end{proof}

Recall the dilation operator $(\vartheta^rf)(x):=f(x',rx_n)$ for $r\in\R\backslash\{0\}$ from \eqref{Eqn::OpDila}.

\begin{prop}\label{Prop::Dila}For any $0<p,q\le\infty$ and $s\in\R$ there is a $M=M_{p,q}^{n,s}>0$ and $C=C_{n,p,q,s,M}>0$ such that $$\|\vartheta^rf\|_{\As_{pq}^s(\R^n)}\le C(|r|^{-M}+|r|^M)\|f\|_{\As_{pq}^s(\R^n)},\quad\text{for }r\neq0,\quad\As\in\{\Bs,\Fs\},\quad f\in\As_{pq}^s(\R^n).$$
% \begin{enumerate}[(i)]
%     \item\label{Item::Dila::1dim} For every $0<p<\infty$, $0<q\le\infty$ and $s>\frac1p-1$ there is a $C_1=C_1(p,q,s)>0$ such that $\|\vartheta^rf\|_{\Fs_{pq}^s(\R)}\le C_1(|r|^{-\frac1p}+|r|^{s-\frac1p})\|f\|_{\Fs_{pq}^s(\R)}$.
%     \item\label{Item::Dila::Inftyq} For every $0<q\le\infty$ and $-1<s<0$ there is a $C_2=C_2(n,q,s)>0$ such that $\|\vartheta^rf\|_{\Fs_{pq}^s(\R^n)}\le C_1(|r|^{??}+|r|^{??})\|f\|_{\Fs_{pq}^s(\R^n)}$.
    
% \end{enumerate}
\end{prop}
% In practice we only need to use $\|\vartheta^rf\|_{\Fs_{pq}^s}\lesssim_M\max(|r|^M,|r|^{-M})\|f\|_{\Fs_{pq}^s}$ for some large enough $M=M(n,p,q,s)>0$.
\begin{proof}We only need to prove the Triebel-Lizorkin case, i.e. $\As=\Fs$. Indeed suppose we get the case $\As=\Fs$, then by real interpolation \eqref{Eqn::Stuff::RealInt1}, for each $0<p,q\le\infty$ and $s\in\R$ there is a $C_{p,q,s}>0$ such that (see \cite[Definition~1.2.2/2 and Theorem~1.3.3]{TriebelInterpolation}, in fact $C_{p,q,s}=C_q$)
\begin{equation}\label{Eqn::RealIntOpNorm}
    \|T\|_{\Bs_{pq}^s(\R^n)\to\Bs_{pq}^s(\R^n)}\le C_{p,q,s}(\|T\|_{\Fs_{pp}^{s-1}(\R^n)\to\Fs_{pp}^{s-1}(\R^n)}+\|T\|_{\Fs_{pp}^{s+1}(\R^n)\to\Fs_{pp}^{s+1}(\R^n)}),
\end{equation}
whenever $T:\Fs_{pp}^{s-1}(\R^n)\to\Fs_{pp}^{s-1}(\R^n)$ is a bounded linear map such that $T:\Fs_{pp}^{s+1}(\R^n)\to\Fs_{pp}^{s+1}(\R^n)$ is also bounded. Taking $T=\vartheta^r$ and replacing  $M_{p,q}^{n,s}$ (obtained from the $\Fs$-cases) with $\max(M_{p,q}^{n,s},M_{p,p}^{n,s-1},M_{p,p}^{n,s+1})$ we complete the proof.

\smallskip
\noindent\textsf{Step 1}: The case $n=1$, $0<p<\infty$, $0<q\le\infty$ and $s>\frac1p-1$. We use \cite[Section 2.3.1]{EdmundsTriebel}. 

Recall $\lambda=(\lambda_j)_{j=0}^\infty$ from Definition~\ref{Defn::BsFs}. When $s>\max(0,1-\frac1p)$, by \cite[Theorem~2.3.3]{TriebelTheoryOfFunctionSpacesII}, we have $\|f\|_{\Fs_{pq}^s(\R)}\approx_{p,q,s}\|f\|_{L^p(\R^n)}+\big\|x\mapsto\big(\int_0^\infty t|(\lambda_1(\frac\cdot t)\ast f)(x)|\frac{dt}{t^{sq+1}}\big)^{1/q}\big\|_{L^p(\R)}$. The same dilation argument from \cite[(2.3.1/4) and (2.3.1/5)]{EdmundsTriebel} shows that
\begin{equation}\label{Eqn::PfDila::1Dim}
    \|\vartheta^rf\|_{\Fs_{pq}^s(\R)}\approx|r|^{-\frac1p}\|f\|_{L^p}+|r|^{s-\frac1p}\Big\|x\mapsto\Big(\int_0^\infty |(\lambda_1(\tfrac\cdot t)\ast f)(x)|\frac{dt}{t^{sq}}\Big)^\frac1q\Big\|_{L^p}\lesssim_{p,q,s}(|r|^{-\frac1p}+|r|^{s-\frac1p})\|f\|_{\Fs_{pq}^s(\R)}.
\end{equation}
In particular we can take $M_{p,q}^{1,s}:=\max(\frac1p,s-\frac1p,\frac1p-s)$ in this case.

\smallskip
\noindent \textsf{Step 2}: The case $n\ge2$, $0<p<\infty$, $0<q\le\infty$ and $s>n\max(0,\frac1p-1,\frac1q-1)$. We use Proposition~\ref{Prop::Stuff} \ref{Item::Stuff::Fubini}:
\begin{align*}
    \|\vartheta^rf\|_{\Fs_{pq}^s(\R^n)}\approx&\Big(\int_{\R^{n-1}}\|\vartheta^r(f(x',\cdot))\|_{\Fs_{pq}^s(\R)}^pdx'\Big)^{1/p}+\Big(\int_{\R}\|f(\cdot,rx_n)\|_{\Fs_{pq}^s(\R^{n-1})}^pdx_n\Big)^{1/p}
    \\
    \lesssim&(|r|^{-\frac1p}+|r|^{s-\frac1p})\Big(\int_{\R^{n-1}}\|f(x',\cdot)\|_{\Fs_{pq}^s(\R)}^pdx'\Big)^{\frac1p}+|r|^{-\frac1p}\Big(\int_{\R}\|f(\cdot,y_n)\|_{\Fs_{pq}^s(\R^{n-1})}^pdy_n\Big)^{\frac1p}
    \\
    \lesssim&(|r|^{-M_{p,q}^{1,s}}+|r|^{M_{p,q}^{1,s}})\|f\|_{\Fs_{pq}^s(\R^n)}.
\end{align*}
In particular we can take $M_{p,q}^{n,s}:=M_{p,q}^{1,s}$ in this case.

\smallskip
\noindent \textsf{Step 3}: The case $0<p<\infty$, $0<q\le\infty$ and $s\in\R$. We use Proposition~\ref{Prop::Stuff} \ref{Item::Stuff::SumDer}.

Let $m\ge0$ be such that $s+m>n\max(0,\frac1p-1,\frac1q-1)$. Note that $\vartheta^\frac1r\circ\partial^\alpha=r^{\alpha_n}\partial^\alpha\circ\vartheta^\frac1r$, thus
\begin{align*}
    &\|\vartheta^rf\|_{\Fs_{pq}^s(\R^n)}\approx\textstyle\inf\Big\{\sum_{|\alpha|\le m}\|g_\alpha\|_{\Fs_{pq}^{s+m}(\R^n)}:\vartheta^rf=\sum_{|\alpha|\le m}\partial^\alpha g_\alpha\Big\}
    \\
    =&\inf\Big\{\sum_{|\alpha|\le m}\|g_\alpha\|_{\Fs_{pq}^{s+m}(\R^n)}:f=\sum_{|\alpha|\le m}r^{\alpha_n}\partial^\alpha (\vartheta^\frac1rg_\alpha)\Big\}=\inf\Big\{\sum_{|\alpha|\le m}\frac{\|\vartheta^rh_\alpha\|_{\Fs_{pq}^{s+m}(\R^n)}}{|r|^{\alpha_n}}:f=\sum_{|\alpha|\le m}\partial^\alpha h_\alpha\Big\}
    \\
    \lesssim&(|r|^{-m}+1)(|r|^{-M_{p,q}^{n,s+m}}+|r|^{M_{p,q}^{n,s+m}})\|f\|_{\Fs_{pq}^s(\R^n)}.
\end{align*}
In particular we can take $M_{p,q}^{n,s}:=M_{p,q}^{n,s+m}+m$ in this case.

\smallskip
\noindent \textsf{Step 4}: The case $p=\infty$, $0<q\le\infty$ and $s\in\R$. We use Proposition~\ref{Prop::Stuff} \ref{Item::Stuff::FInftyQ}.

For $x\in\R^n$ and $R>0$, we see that $f\in\Fs_{pq}^s\big(B\big((x',\frac{x_n}r),\max(R,\frac R{|r|})\big)\big)$ implies $\vartheta^r f\in\Fs_{pq}^s(B(x,R))$. Therefore
\begin{align*}
    &\|\vartheta^rf\|_{\Fs_{\infty q}^s(\R^n)}\approx\sup_{x\in\R^n,R>0}R^{-\frac nq}\inf\big\{\|\tilde g\|_{\Fs_{qq}^s(\R^n)}:\tilde g|_{B(x,R)}=(\vartheta^rf)|_{B(x,R)}\big\}
    \\
    \le&\sup_{x\in\R^n,R>0}R^{-\frac nq}\inf\big\{\|\vartheta^\frac1r\tilde f\|_{\Fs_{qq}^s(\R^n)}:\tilde f\big|_{B\big((x',\frac{x_n}r),\max(R,\frac R{|r|})\big)}=f\big|_{B\big((x',\frac{x_n}r),\max(R,\frac R{|r|})\big)}\big\}
    \\
    \lesssim&(|r|^{-M_{q,q}^{n,s}}+|r|^{M_{q,q}^{n,s}})\sup_{x\in\R^n,R>0}R^{-\frac nq}\| f\|_{\Fs_{qq}^s\left(B\left((x',\frac{x_n}r),\max(R,\frac R{|r|})\right)\right)}
    \\
    \lesssim&(|r|^{-M_{q,q}^{n,s}}+|r|^{M_{q,q}^{n,s}})(1+|r|^{-\frac nq})\sup_{x\in\R^n,R>0}R^{-\frac nq}\| f\|_{\Fs_{qq}^s(B(x,R))}\lesssim(|r|^{-M_{q,q}^{n,s}-\frac nq}+|r|^{M_{q,q}^{n,s}})\| f\|_{\Fs_{\infty q}^s(\R^n)}.
\end{align*}
Taking $M_{\infty,q}^{n,s}:=M_{q,q}^{n,s}+n/q$, we complete the proof.
% For \ref{Item::Dila::Inftyq} we use \cite{DrihemTLTypeDiff}. By \cite[Theorem~4.6]{DrihemTLTypeDiff} we have, for every $0<q\le\infty$ and $0<s<1$,
% \begin{equation*}
%     \|f\|_{\Fs_{\infty q}^s(\R^n)}
% \end{equation*}
\end{proof}

We now start to prove Theorem~\ref{MainThm} \ref{Item::Thm::BFBdd}. In order to handle Theorem~\ref{MainThm} \ref{Item::Thm::MorreyBdd} later, we consider a more quantitative version of Theorem~\ref{MainThm} \ref{Item::Thm::BFBdd}. It will contain an estimate of the operator
norm of $E$.

\begin{thm}[Quantitative common extensions]\label{Thm::QntBFBdd}
For every $0<p,q\le\infty$, $s\in\R$, there is a constant $\widehat M=\widehat M_{p,q}^{n,s}$, such that the following holds:

Let $m_1,m_2\in\Z_{\ge0}$ be two non-negative integers such that $n\max(0,\frac1p-1,\frac1q-1)-m_1<s<\frac1p+m_2$. %and for every $0<p,q\le\infty$, $n\max(0,\frac1p-1,\frac1q-1)-m_1<s<\frac1p+m_2$, a constant $C_{n,p,q,s}>0$, 
Let $a=(a_j)_{j=-\infty}^\infty\subset\R$ and $b=(b_j)_{j=-\infty}^\infty\subset\R_+$ be two sequences satisfying
\begin{equation}\label{Eqn::QntBFBdd::CondAB}
    \sum_{j=-\infty}^\infty a_j(-b_j)^k=1\text{ for }-m_1\le k\le m_2;\quad \sum_{j\in\Z} 2^{\delta|j|}|a_j|(b_j^{-\max(m_1,\widehat M)}+b_j^{\max(m_2,\widehat M)})<\infty\text{ for some }\delta>0.
\end{equation}
Then $E=E^{a,b}$ given by \eqref{Eqn::ExtOp} has the boundedness $E:\As_{pq}^s(\R^n_+)\to\As_{pq}^s(\R^n)$ for $\As\in\{\Bs,\Fs\}$. Moreover for the operator norms of $E$ we have
\begin{equation}\label{Eqn::QntBFBdd::ENorm}
    \|Ef\|_{\As_{pq}^s(\R^n)}\le C_{n,p,q,s,\delta,{\widehat M}}\sum_{j\in\Z}2^{\delta|j|}|a_j|\big(b_j^{-\widehat M_{p,q}^{n,s}}+b_j^{\widehat M_{p,q}^{n,s}}\big)\|f\|_{\As_{pq}^s(\R^n_+)}.
\end{equation}
Here $C_{n,p,q,s,\delta,\widehat M}$ does not depend on the choices of $(a_j)_j$ and $(b_j)_j$.
\end{thm}
\begin{rem}
    By Proposition~\ref{Prop::ConsAB}, the sequences $(a_j)_j$ and $(b_j)_j$ can be chosen independently of $p,q,n,s,m_1,m_2$. Note that $\delta$ in \eqref{Eqn::QntBFBdd::CondAB} is allowed to depend on $m_1,m_2,\widehat M_{pq}^{ns}$ (cf. the $\delta_k$ in \eqref{Eqn::CondAB} and Remark~\ref{Rmk::DeltaK}). 
    
    Therefore, letting $m_1,m_2\to+\infty$ we obtain all boundedness properties of $E$ in Theorem~\ref{MainThm} \ref{Item::Thm::BFBdd}.
    
    Additionally, Theorem \ref{Thm::QntBFBdd} also gives the following:
    % For operator norms \eqref{Eqn::QntBFBdd::ENorm}, one can see that $\widehat M_{p,q}^{n,s}$ can be chosen to be independent of $m_1,m_2$, as one can always that $m_1,m_2$ to be the smallest so that $n\max(0,\frac1p-1,\frac1q-1)-m_1<s<\frac1p+m_2$ is satisfies
    
\begin{itemize}
    \item If one take $m_1=0$ and let $m_2\to+\infty$, we see that Seeley's extension operator (see Remark \ref{Rmk::SeeleyCont1}) has boundedness $E:\As_{pq}^s(\R^n_+)\to\As_{pq}^s(\R^n)$ for $\As\in\{\Bs,\Fs\}$, $p,q\in(0,\infty]$ and $s>n\max(0,\frac1p-1,\frac1q-1)$. It is possible that the range for $s$ is not optimal.
    \item If one take $(a_j)_j$ to be finitely nonzero (depending on the upper bound of $m_1,m_2$), then the qualitative result $E:\As_{pq}^s(\R^n_+)\to\As_{pq}^s(\R^n)$ is well-known for \textit{common extension operators}. See \cite[Remark~2.72]{TriebelTheoryOfFunctionSpacesIV}.
\end{itemize}    
\end{rem}

\begin{proof}[Proof of Theorem~\ref{Thm::QntBFBdd} (hence Theorem~\ref{MainThm} \ref{Item::Thm::BFBdd})]Similar to the proof of Proposition~\ref{Prop::Dila}, using \eqref{Eqn::RealIntOpNorm} we only need to prove the case $\As=\Fs$.

To make the notation clear, we use $E_n^{a,b}$ for the extension operator of $\R^n_+$ given in \eqref{Eqn::ExtOp} associated with sequences $a=(a_j)_{j=-\infty}^\infty$ and $b=(b_j)_{j=-\infty}^\infty$.
Recall the zero extension operator $S_nf(x):=\1_{\R^n_+}(x)f(x)$ from \eqref{Eqn::OpS}.
 
% $Sf(t):=f(t)$ if $t>0$ and $Sf(t):=0$ if $t<0$. For $n\ge1$ and $(a,b)=(a_j,b_j)_{j=-\infty}^\infty\subset\R$ that satisfies \eqref{Eqn::CondAB}, we write $E^{a,b}_n$ for the extension operator of $\R^n_+$. % For $r\in\R\backslash\{0\}$ we define the scaling map $\vartheta^rf(t):=f(rt)$. By ?? we have $\|\vartheta^r f\|_{\Fs_{pq}^s(\R)}\lesssim\max(1,r^{s-1/p})\|f\|_{\Fs_{pq}^s(\R)}$ where the implied constant depends only on the $\lambda$ in ??.
% \begin{equation}
%     (E_n^{a,b}f)(x',x_n)=E_1^{a,b}(f(x',\cdot))(x_n),\quad\text{and }\partial^\alpha E_n^{a,b}f=E_n^{a(-b)^{\alpha_n},b}\partial^\alpha f\text{ where }a(-b)^{\alpha_n}=(a_j(-b_j)^{\alpha_n})_{j=-\infty}^\infty.
% \end{equation}

\medskip
\noindent\textsf{Case A}: We consider $p<\infty$ and $s$ suitably large. We sub-divide the discussion into three parts.

\smallskip
\hypertarget{HTag::PfQStep1}{\textsf{Step A1}}: The case $n=1$, $0<p<\infty$, $0<q\le\infty$ and $\frac1p-1<s<\frac1p$, with arbitrary $m_1,m_2\ge0$.

Indeed by \cite[Theorem~2.48]{TriebelTheoryOfFunctionSpacesIV} we have boundedness $S_1:\Fs_{pq}^s(\R_+)\to\Fs_{pq}^s(\R)$ at this range. Clearly $E^{a,b}_1=S_1+\sum_{j=-\infty}^\infty a_j\cdot\vartheta^{-b_j}\circ S_1$. 
Therefore by Lemma~\ref{Lem::TriSum} and Proposition~\ref{Prop::Dila},
\begin{align*}
    \|E^{a,b}_1f\|_{\Fs_{pq}^s(\R)}\lesssim&_{\min(p,q,1),\delta}\|S_1f\|_{\Fs_{pq}^s(\R)}+\sum_{j\in\Z}2^{\delta|j|}|a_j|\|\vartheta^{-b_j}S_1f\|_{\Fs_{pq}^s(\R)}
    \\
    \lesssim&_{p,q,s}\sum_{j\in\Z}2^{\delta|j|}|a_j|\big(b_j^{-M_{pq}^{1s}}+b_j^{M_{pq}^{1s}}\big)\|S_1f\|_{\Fs_{pq}^s(\R)}
    \lesssim_{p,q,s}\sum_{j\in\Z}2^{\delta|j|}|a_j|\big(b_j^{-M_{pq}^{1s}}+b_j^{M_{pq}^{1s}}\big)\|f\|_{\Fs_{pq}^s(\R_+)}.
\end{align*}
% Here we note that $\sum_{j\in\Z}|a_j|\ge1$ and $b_j^{\frac1p-s}+b_j^{s-\frac1p}\ge1$. 
We conclude the case by taking $\widehat M_{pq}^{1s}:=M_{pq}^{1s}$, the constant in Proposition~\ref{Prop::Dila}.

\smallskip
\hypertarget{HTag::PfQStep2}{\textsf{Step A2}}:  The case $n=1$, $0<p<\infty$, $0<q\le\infty$ and $s>\frac1p-1$, with $m_1\ge0$ and $m_2>s-\frac1p$.

By \cite[Theorem~2.3.8 and Remark~3.3.5/2]{TriebelTheoryOfFunctionSpacesI} $\|f\|_{\Fs_{pq}^s(U)}\approx_{p,q,s,m_0}\sum_{k=0}^{m_0}\|\partial^kf\|_{\Fs_{pq}^{s-m_0}(U)}$ holds for every $m_0\ge1$ and $U\in\{\R^n,\R^n_+\}$. Recall from Remark~\ref{Rmk::RmkofTrivialComm} that $\partial^k\circ E_1^{a,b}=E_1^{a(-b)^k,b}\circ\partial^k$ holds for $0\le k\le m_2$. 

When $s-\frac1p\notin\Z$, we let $m_0:=\lceil s-\frac1p\rceil$. Therefore by \hyperlink{HTag::PfQStep1}{Step A1} (since $m_2\ge m_0$),
\begin{align*}
    \|E^{a,b}_1f\|_{\Fs_{pq}^s(\R)}\approx&_{p,q,s}\sum_{k=0}^{m_0}\|E^{a(-b)^k,b}_1\partial^kf\|_{\Fs_{pq}^{s-m_0}}
    \lesssim_{p,q,s,\delta}\sum_{k=0}^{m_0}\sum_{j\in\Z}2^{\delta|j|}|a_j||\big(b_j^{k-\widehat M_{pq}^{1,s-m_0}}+b_j^{k+\widehat M_{pq}^{1,s-m_0}}\big)\|\partial^kf\|_{\Fs_{pq}^{s-m_0}}
    \\
    \lesssim&_{m_0}\sum_{j\in\Z}2^{\delta|j|}|a_j||\big(b_j^{-\widehat M_{pq}^{1,s-m_0}}+b_j^{m_0+\widehat M_{pq}^{1,s-m_0}}\big)\|f\|_{\Fs_{pq}^s(\R_+)}.
\end{align*}
In this case we can take $\widehat M_{pq}^{1s}:=\widehat M_{pq}^{1,s-m_0}+m_0=\widehat M_{pq}^{1,s-\lceil s-\frac1p\rceil}+\lceil s-\frac1p\rceil$.

The case $s-\frac1p\in\Z_+$ follows from the $\pm$ interpolation \eqref{Eqn::Stuff::CpxInt}, where we can take $\widehat M_{pq}^{1s}:=\max(\widehat M_{pq}^{1,s-\frac12},\widehat M_{pq}^{1,s+\frac12})$. %with $m_0:=\lceil s-\frac1p\rceil$.

\smallskip
\hypertarget{HTag::PfQStep3}{\textsf{Step A3}}: The case $n\ge1$, $0<p<\infty$, $0<q\le\infty$ and $s>n\max(0,\frac1p-1,\frac1q-1)$, with $m_1\ge0$ and $m_2>s-\frac1p$.

Indeed we have $(E_n^{a,b}f)(x)=E_1^{a,b}(f(x',\cdot))(x_n)$. Therefore, for any extension $\tilde f\in\Fs_{pq}^s(\R^n)$ of $f\in\Fs_{pq}^s(\R^n_+)$,
\begin{align*}
    &\|E_n^{a,b}f\|_{\Fs_{pq}^s(\R^n)}
    \\
    \approx&_{p,q,s}\Big(\int_{\R^{n-1}}\|E_1^{a,b}(f(x',\cdot))\|_{\Fs_{pq}^s(\R)}^pdx'\Big)^{1/p}+\Big(\int_{\R}\|(E_n^{a,b}f)(\cdot,x_n)\|_{\Fs_{pq}^s(\R^{n-1})}^pdx_n\Big)^{1/p}\hspace{-0.3in}&(\text{by Proposition~\ref{Prop::Stuff} \ref{Item::Stuff::Fubini}})
    \\
    \le% 要不要加个下标% Nope
    &\|E_1^{a,b}\|_{\Fs_{pq}^s}\Big(\int_{\R^{n-1}}\|f(x',\cdot)\|_{\Fs_{pq}^s(\R_+)}^pdx'\Big)^{1/p}+ \|E_1^{a,b}f\|_{L^p(\R_{x_n};\Fs_{pq}^s(\R^{n-1}_{x'}))}
    % \sum_{j=-\infty}^\infty 2^{\delta|j|}|a_j|\Big(\int_{\R_+}\|f(\cdot,-b_jx_n)\|_{\Fs_{pq}^s(\R^{n-1})}^pdx_n\Big)^{\frac1p}
    \\
    \lesssim&_{p,q,s}\|E_1^{a,b}\|_{\Fs_{pq}^s}\|f\|_{L^p(\R^{n-1};\Fs_{pq}^s(\R_+))}+\|E_1^{a,b}\|_{L^p(\Fs_{pq}^s)}\|f\|_{L^p(\R_+;\Fs_{pq}^s(\R^{n-1}))}
    \\
    \le&_{p,q,s,\delta}\sum_{j\in\Z}2^{\delta|j|}|a_j|\big(b_j^{-\widehat M_{pq}^{1s}}+b_j^{\widehat M_{pq}^{1s}}\big)\|\tilde f\|_{L^p(\R^{n-1};\Fs_{pq}^s)}+\sum_{j\in\Z}2^{\delta|j|}|a_j|b_j^{-\frac1p}\|\tilde f\|_{L^p(\R;\Fs_{pq}^s)}&(\text{by \hyperlink{HTag::PfQStep2}{Step A2} and Remark~\ref{Rmk::PfSobBdd::VectValuedTmp}})
    \\
    \lesssim&_{p,q,s}\sum_{j\in\Z}2^{\delta|j|}|a_j|\big(b_j^{-\widehat M_{pq}^{1s}-\frac1p}+b_j^{\widehat M_{pq}^{1s}}\big)\|\tilde f\|_{\Fs_{pq}^s(\R^n)}&(\text{by Proposition~\ref{Prop::Stuff} \ref{Item::Stuff::Fubini}}).
\end{align*}
%Here the operator norm on vector-valued function space $\|E_1^{a,b}\|_{L^p(\Fs_{pq}^s)}\lesssim_{p,q,\delta}\sum_{j\in\Z}2^{\delta|j|}|a_j|b_j^{-\frac1p}$ comes from Remark~\ref{Rmk::PfSobBdd::VectValuedTmp}.

Taking the infimum over all extensions $\tilde f$ of $f$, we get the claim with $\widehat M_{pq}^{ns}:=\widehat M_{pq}^{1s}+1/p$.

\medskip
\noindent\textsf{Case B}: We consider $p=\infty$, i.e. the $\Fs_{\infty q}^s$-cases, for $s>-1$.

\smallskip
\hypertarget{HTag::PfQStep1'}{\textsf{Step B1}}: The case $n\ge1$, $p=\infty$, $0<q\le\infty$ and $-1<s<0$, with arbitrary $m_1,m_2\ge0$.

The proof is identical to \hyperlink{HTag::PfQStep1}{Step A1}, where we use $E^{a,b}_n=S_n+\sum_{j=-\infty}^\infty a_j\cdot\vartheta^{-b_j}\circ S_n$ and $S_n:\Fs_{pq}^s(\R^n_+)\to\Fs_{pq}^s(\R^n)$ (see \cite[Theorem~2.48]{TriebelTheoryOfFunctionSpacesIV}). We get the result with $\widehat M_{\infty q}^{ns}:=M_{\infty q}^{ns}$, the constant in Proposition~\ref{Prop::Dila}.

\smallskip
\hypertarget{HTag::PfQStep2'}{\textsf{Step B2}}:  The case $n\ge1$, $p=\infty$, $0<q\le\infty$ and $s>-1$, with $m_1\ge0$ and $m_2>s$.

It remains to prove the case $s\ge0$. The proof is similar to \hyperlink{HTag::PfQStep2}{Step A2}, except we use (see \cite[Theorem~1.24]{TriebelTheoryOfFunctionSpacesIV}) $\|f\|_{\Fs_{\infty q}^s(\R^n)}\approx_{p,q,s,m_0}\sum_{|\gamma|\le m_0}\|\partial^\gamma f\|_{\Fs_{pq}^{s-m_0}(\R^n)}$. When $s\notin\Z_+$, since $m_2\ge\lceil s\rceil$ we get from \hyperlink{HTag::PfQStep1'}{Step B1} that $$\|E_n^{a,b}f\|_{\Fs_{pq}^s(\R^n)}\lesssim_{q,s,\delta}\sum_{j\in\Z}2^{\delta|j|}|a_j||\big(b_j^{-\widehat M_{\infty q}^{n,s-\lceil s\rceil}}+b_j^{\lceil s\rceil+\widehat M_{\infty q}^{n,s-\lceil s\rceil}}\big)\|\tilde f\|_{\Fs_{\infty q}^s(\R^n)},\quad\text{ whenever }\tilde f|_{\R^n_+}=f.$$
Taking infimum of the extensions $\tilde f$ over $f$ we get the claim with $\widehat M_{\infty q}^{ns}:=\widehat M_{\infty q}^{n,s-\lceil s\rceil}+\lceil s\rceil$.

The case $s\in\Z_+$ follows from the $\pm$ interpolation \eqref{Eqn::Stuff::CpxInt}, where we can take $\widehat M_{\infty q}^{ns}:=\max(\widehat M_{\infty q}^{n,s-\frac12},\widehat M_{\infty q}^{n,s+\frac12})$.

\smallskip
\noindent\hypertarget{HTag::PfQStep4}{\textsf{Final step}}: The general case: $n\ge1$, $0<p,q\le\infty$ and $s\in\R$, with $m_1>n\max(0,\frac1p-1,\frac1q-1)-s$ and $m_2>s-\frac1p$. Here we can take $m_1:=\lfloor n\max(0,\frac1p-1,\frac1q-1)-s\rfloor+1$.

By Remark~\ref{Rmk::RmkofTrivialComm} again we have $E_n^{a,b}\circ\partial^\gamma=\partial^\gamma\circ E_n^{a(-b)^{-\gamma_n},b}$ for all $|\gamma|\le m_1$, where $a(-b)^{-\gamma_n}=(a_j(-b_j)^{-\gamma_n})_{j=-\infty}^\infty$. Therefore when $f\in\Fs_{pq}^s(\R^n_+)$ and $g_\gamma\in\Fs_{pq}^{s+m_1}(\R^n_+)$ satisfies $f=\sum_{|\gamma|\le m_1}\partial^\gamma g_\gamma$, by Steps \hyperlink{HTag::PfQStep3}{A3} and \hyperlink{HTag::PfQStep3'}{B2} we have
\begin{equation}
\begin{aligned}
    &\|E_n^{a,b}f\|_{\Fs_{pq}^s(\R^n)}=\Big\|\sum_{|\gamma|\le m_1}\partial^\gamma\circ E_n^{a(-b)^{-\gamma_n},b}g_\gamma\Big\|_{\Fs_{pq}^s(\R^n)}
    \\
    \lesssim&_{\min(p,q,1),m_1}\sum_{|\gamma|\le m_1}\|\partial^\gamma\circ E_n^{a(-b)^{-\gamma_n},b}g_\gamma\|_{\Fs_{pq}^s(\R^n)}
    \\
    \lesssim&_{p,q,s,m_1}\sum_{|\gamma|\le m_1}\|E_n^{a(-b)^{-\gamma_n},b}g_\gamma\|_{\Fs_{pq}^{s+m_1}(\R^n)}&(\text{see \cite[(2.3.8/3)]{TriebelTheoryOfFunctionSpacesI}})
    \\
    \lesssim&_{p,q,s,m_1,\delta}\sum_{|\gamma|\le m_1}\sum_{j\in\Z}2^{\delta|j|}|a_j|\big(b_j^{-\widehat M_{pq}^{n,s+m_1}-\gamma_n}+b_j^{\widehat M_{pq}^{n,s+m_1}-\gamma_n}\big)\|g_\gamma\|_{\Fs_{pq}^{s+m_1}(\R^n)}&(\text{by Steps \hyperlink{HTag::PfQStep3}{A3} and \hyperlink{HTag::PfQStep3'}{B2}}).
\end{aligned}
\end{equation}
Taking the infimum over the decompositions $f=\sum_{|\alpha|\le m}\partial^\alpha g_\alpha$, by Proposition~\ref{Prop::Stuff} \ref{Item::Stuff::SumDer} we obtain the claim with $\widehat M_{pq}^{ns}:=\widehat M_{pq}^{n,s+m_1}+m_1=\widehat M_{pq}^{n,s+\lfloor n\max(0,\frac1p-1,\frac1q-1)-s\rfloor+1}+\lfloor n\max(0,\frac1p-1,\frac1q-1)-s\rfloor+1$. This completes the proof of $\As=\Fs$. The case $\As=\Bs$ is done by real interpolations, as discussed in the beginning of the proof.
\end{proof}

% For completeness we include the discussion of $E^*$.
% \begin{proof}[Proof of Theorem~\ref{MainThm} \ref{Item::Thm::BFBdd}]
% The boundedness $E:\As_{pq}^s(\R^n_+)\to\As_{pq}^s(\R^n)$ is the qualitative version of Theorem~\ref{Thm::QntBFBdd}.

% To prove the boundedness of $E^*$ we use the similar argument to \eqref{Eqn::PfSobBdd::E*}, where for $g\in\Ss'(\R^n)$ we see that $g+\sum_{j\in\Z}\frac{a_j}{b_j}\delta^{-1/b_j}g$ is the extension of $E^*g\in\Ss'(\R^n)$. Therefore by Proposition~?????? 
% \end{proof}

We now prove Theorem \ref{MainThm} \ref{Item::Thm::MorreyBdd}, the boundedness on   $\Bs_{pq}^{s\tau}$, $\Fs_{pq}^{s\tau}$ and $\Ns_{pq}^{s\tau}$, using Theorem \ref{Thm::QntBFBdd}.

In the following, for $m\in\Z_+$ and $r>0$, we define a sequence $A^{mr}=(A^{mr}_j)_{j=0}^{2m}$ to be the unique sequence such that $\sum_{j=0}^{2m}A^{mr}_j(-2^{-j}r)^k=1$ for integers $-m\le k\le m$. Recall from Remark~\ref{Rmk::SeeleyCont2} we have 
\begin{equation}
    A^{mr}_j=(-r2^{-j})^m\prod_{0\le k\le 2m;k\neq j}\frac{r2^{-k}+1}{r2^{-k}-r2^{-j}}=(-1)^m\frac{2^{-jm}}{r^m}\prod_{0\le k\le 2m;k\neq j}\frac{r2^{-k}+1}{2^{-k}-2^{-j}},\quad 0\le j\le 2m.
\end{equation}
We see that $|A^{mr}_j|\lesssim_m r^m+r^{-m}$ for $0\le j\le 2m$, where the implicit constant does not depend on $r>0$.

We define
\begin{equation}\label{Eqn::ExtEb}
    \Eb^{m,r}f(x',x_n)=\Eb^{m,r}_nf(x',x_n):=\begin{cases}\sum_{j=0}^{2m}A^{mr}_jf(x',-r2^{-j}x_n)&x_n<0\\f(x)&x_n>0\end{cases}.
\end{equation}

By Theorem~\ref{Thm::QntBFBdd}, for every $0<p,q\le\infty$, $s\in\R$ and % times n
$m>\max(s-\frac1p,-s,n(\frac1p-1)-s,n(\frac1q-1)-s)$, there is a $C=C_{n,p,q,s,m}>0$, such that $\Eb^{m,r}:\As_{pq}^s(\R^n_+)\to\As_{pq}^s(\R^n)$ is bounded for $\As\in\{\Bs,\Fs\}$ and for all $r>0$, with
\begin{equation}\label{Eqn::OpNormEb}
    \|\Eb^{m,r}\|_{\As_{pq}^s(\R^n_+)\to\As_{pq}^s(\R^n)}\le C_{n,p,q,s,m}(r^{-m\cdot\widehat M_{pq}^{ns}}+r^{m\cdot\widehat M_{pq}^{ns}}),\quad\forall\, r>0.
\end{equation}
Here $\widehat M_{pq}^{ns}>0$ is the corresponding constant in Theorem~\ref{Thm::QntBFBdd}.

\begin{proof}[Proof of Theorem~\ref{MainThm} \ref{Item::Thm::MorreyBdd}]
We only need to prove the case $\As\in\{\Bs,\Fs\}$. Indeed, when $\tau<\frac1p$ (hence $p<\infty$), the boundedness of $\Ns_{pq}^{s\tau}$ can be deduced from real interpolations \eqref{Eqn::Stuff::RealInt2}. When $\tau=\frac1p$ we have coincidence $\Ns_{pq}^{s\frac1p}=\Bs_{\infty q}^s$ (see \eqref{Eqn::EqvSpaces2}), which is treated in Theorem~\ref{MainThm} \ref{Item::Thm::BFBdd}.

We begin with the case $s<0$.

Recall from Proposition~\ref{Prop::Stuff} \ref{Item::Stuff::FInftyQ}, $\|\tilde f\|_{\As_{pq}^{s\tau}(\R^n)}\approx_{n,p,q,s,\tau}\sup_{x\in\R^n;R>0}R^{-n\tau}\|\tilde f\|_{\As_{pq}^s(B(x,R))}$ for $\As\in\{\Bs,\Fs\}$ ($p<\infty$ if $\As=\Fs$) under the assumption $s<0$.

Clearly if $x_n\ge R$, then $\|Ef\|_{\As_{pq}^s(B(x,R))}=\|f\|_{\As_{pq}^s(B(x,R))}$. 

If $x_n\le -R$, then $(E^{a,b}f)|_{B(x,R)}=\sum_{j=-\infty}^\infty a_j\cdot\vartheta^{-b_j}f$.  Therefore, for every extension $\tilde f\in\As_{pq}^{s\tau}(\R^n)$ of $f$,
\begin{equation}\label{Eqn::PfMorreyBdd::DilSum}
\begin{aligned}
    &R^{-n\tau}\|E^{a,b}f\|_{\As_{pq}^s(B(x,R))}\lesssim_{p,q,\delta}R^{-n\tau}\sum_{j=-\infty}^\infty 2^{\delta|j|} |a_j|\|\vartheta^{-b_j}f\|_{\As_{pq}^s(B(x,R))}&(\text{by Lemma~\ref{Lem::TriSum}})
    \\
    \le &R^{-n\tau}\sum_{j=-\infty}^\infty 2^{\delta|j|} |a_j|(b_j^{-M_{pq}^{ns}}+b_j^{M_{pq}^{ns}})\|f\|_{\As_{pq}^s\left(B\left((x',-\frac{x_n}{b_j}),\max(R,\frac R{b_j})\right)\cap\R^ n_+\right)}&(\text{by Proposition~\ref{Prop::Dila}})
    \\
    \le&\sum_{j=-\infty}^\infty 2^{\delta|j|} |a_j|(b_j^{-M_{pq}^{ns}}+b_j^{M_{pq}^{ns}})\max(R,\tfrac R{b_j})^{-n\tau}\sup_{y\in\R^n}\|\tilde f\|_{\As_{pq}^s\big(B(y,\max(R, R/{b_j}))\big)}%\hspace{-0.1in}
    \\
    \lesssim&_{n,p,q,s,a,b}\sup_{y\in\R^n;\rho>0}\rho^{-n\tau}\|\tilde f\|_{\As_{pq}^s(B(y,\rho))}\approx_{n,p,q,s,\tau} \|\tilde f\|_{\As_{pq}^{s\tau}(\R^n)}&(\text{by Proposition~\ref{Prop::Stuff} \ref{Item::Stuff::FInftyQ}}).
\end{aligned}    
\end{equation}

Here $M_{pq}^{ns}$ is the constant in Proposition~\ref{Prop::Dila}. Taking the infimum of the extension $\tilde f$ over $f$ we see that 
$$R^{-n\tau}\|Ef\|_{\As_{pq}^s(B(x,R))}\lesssim_{n,p,q,s,\tau,a,b}\|f\|_{\As_{pq}^{s\tau}(\R^n_+)}\text{ uniformly for }R>0\text{ and }x_n\le -R,\text{ and hence for all }|x_n|\ge R.$$

\smallskip
The difficult part is to prove the estimate for $|x_n|\le R$. In this case $B(x,R)\subset B((x',0),2R)$, thus $R^{-n\tau}\|Ef\|_{\As_{pq}^s(B(x,R))}\le2^{n\tau}(2R)^{-n\tau}\|Ef\|_{\As_{pq}^s(B((x',0),2R))}$. Since $E$ is translation invariant in $x'$, we can assume $x'=0$ without loss of generality. 

To summarize the discussion above, to prove the boundedness $E:\As_{pq}^{s\tau}(\R^n_+)\to\As_{pq}^{s\tau}(\R^n)$ for $0<p,q\le\infty$, $s<0$ and $0\le\tau\le\frac1p$ ($p<\infty$ if $\As=\Fs$), it remains to show the existence of $C=C_{n,p,q,s,\tau,a,b}>0$ such that
\begin{equation}\label{Eqn::PfMorreyBdd::Key}
    R^{-n\tau}\|Ef\|_{\As_{pq}^s(B(0,R))}\le C\sup_{\rho>0}\rho^{-n\tau}\|\tilde f\|_{\As_{pq}^s(B(0,\rho))},\quad\forall\, R>0,\quad\forall\, \tilde f\in\As_{pq}^s(\R^n)\text{ such that }\tilde f|_{\R^n_+}=f.
\end{equation}

% Here by \eqref{Eqn::CondAB} the sum $\sum_{j=-\infty}^\infty 2^{\delta|j|} |a_j|(b_j^{-M_{pq}^{ns}-n\tau}+b_j^{M_{pq}^{ns}+n\tau})<\infty$. Thus $R^{-n\tau}\|E_n^{a,b}f\|_{\As_{pq}^s(B(x,R))}\lesssim\|f\|_{\As_{pq}^{s\tau}(\R^n_+)}$.

In the following we let 
$$B_+(0,R):=B(0,R)\cap\R^n_+;\quad J_+:=\{j\in\Z:b_j>1\},\quad J_-:=\{j\in\Z:b_j\le 1\},\quad\text{thus}\quad\Z=J_+\amalg J_-.$$ 

We fix a $R>0$ temporally.

For each $j\in J_+$, take an extension $\tilde f_j=\tilde f_j^R\in\As_{pq}^s(\R^n)$ of $f|_{B_+(0,b_jR)}$ such that $\|\tilde f_j\|_{\As_{pq}^s(\R^n)}\le2\|f\|_{\As_{pq}^s(B_+(0,b_jR))}$. We let $\tilde f_-=\tilde f_-^R\in\As_{pq}^s(\R^n)$ be an extension of $f|_{B_+(0,R)}$ such that $\|\tilde f_-\|_{\As_{pq}^s(\R^n)}\le2\|f\|_{\As_{pq}^s(B_+(0,R))}$. % We denote $\tilde f_j:=\tilde f_-$ for $j\in J_-$.

Clearly we have the following extension of $Ef|_{B(0,R)\cap\R^n_-}$:
$$\tilde f_{<0}=\tilde f_{<0}^R:=\sum_{j\in J_+} a_j\cdot\vartheta^{-b_j}\tilde f_j+\sum_{j\in J_-}a_j\cdot\vartheta^{-b_j}\tilde f_-.$$
However $\tilde f_{<0}$ is not the extension of $f|_{B_+(0,R)}=Ef|_{B_+(0,R)}$. To modify it on the domain $B_+(0,R)$ we consider 
\begin{equation*}
    \tilde f_{>0}=\tilde f_{>0}^R:=\tilde f_-+\sum_{j\in J_+}\vartheta^{-1}\circ\Eb^{m,\frac1{b_j}}\big[\vartheta^{b_j}(\tilde f_j-\tilde f_-)\big]=\tilde f_-+\sum_{j\in J_+}\sum_{k=0}^{2m}A^{m,\frac1{b_j}}_k\cdot(\vartheta^{2^{-k}}\tilde f_j-\vartheta^{2^{-k}}\tilde f_-)\quad\text{in }\R^n_+,
\end{equation*}
where % times n
we have chosen $m>\max(s-\frac1p,-s,n(\frac1p-1)-s,n(\frac1q-1)-s)$, and $\Eb^{m,1/b_j}$ is given by \eqref{Eqn::ExtEb}.

When $j\in J_+$ (i.e $b_j>1$) we have $(\tilde f_j-\tilde f_-)|_{B(0,R)}=0$. Thus in this case $\Eb^{m,\frac1{b_j}}[\vartheta^{-b_j}(\tilde f_j-\tilde f_-)]\big|_{B(0,R)}\equiv0$ (since $2^{-k}/b_j<1$ for $j\in J_+$ and $k\ge0$). We conclude that $\tilde f_{>0}|_{B_+(0,R)}=\tilde f_-|_{B_+(0,R)}=f|_{B_+(0,R)}$.

On the other hand we can rewrite $\tilde f_{<0}$ as 
\begin{equation}
    \tilde f_{<0}=\sum_{j=-\infty}^\infty a_j\cdot\vartheta^{-b_j}\tilde f_-+\sum_{j\in J_+}a_j\cdot\vartheta^{-b_j}(\tilde f_j-\tilde f_-)=E^{a,b}\tilde f_-+\sum_{j\in J_+}a_j\cdot\vartheta^{-1}\circ\vartheta^{b_j}(\tilde f_j-\tilde f_-)\quad\text{on }\R^n_-.
\end{equation}
Here for global defined function $\tilde g$, we use the notation $E\tilde g$ for $E[(\tilde g|_{\R^n_+})]$.

Therefore, the function $\tilde F=\tilde F^R:=\big(\tilde f_{>0}|_{\R^n_+}\big)\cup\big(\tilde f_{<0}|_{\R^n_-}\big)$ is  an extension of $Ef|_{B(0,R)}$ and has the expression
$$\tilde F=E^{a,b}\big[\tilde f_-\big]+\sum_{j\in J_+}a_j\cdot\vartheta^{-1}\circ\Eb^{m,\frac1{b_j}}\big[\vartheta^{b_j}(\tilde f_j-\tilde f_-)\big]\in\Ss'(\R^n).$$

We see that
\begin{align*}
    \|\tilde F\|_{\As_{pq}^s(\R^n)}\lesssim&_{p,q,\delta}\|E^{a,b}\tilde f_-\|_{\As_{pq}^s(\R^n)}+\sum_{j\in J_+}2^{\delta|j|}|a_j|\|\Eb^{m,\frac1{b_j}}\|_{\As_{pq}^s}\|\vartheta^{b_j}(\tilde f_j-\tilde f_-)\|_{\As_{pq}^s(\R^n)}\hspace{-0.2in}&(\text{by Lemma~\ref{Lem::TriSum}})
    \\
    \lesssim&_{n,p,q,s,m}C'_{\delta, a, b,\widehat M}\|\tilde f_-\|_{\As_{pq}^s(\R^n_+)}+\sum_{j\in J_+}2^{\delta|j|}|a_j|b_j^{m\cdot\widehat M_{pq}^{ns}}\|\vartheta^{b_j}(\tilde f_j-\tilde f_-)\|_{\As_{pq}^s}\hspace{-0.05in}&(\text{by Theorem~\ref{Thm::QntBFBdd} and \eqref{Eqn::OpNormEb}})
    \\
    \lesssim&_{n,p,q,s,a,b,\delta,\widehat M}\|\tilde f_-\|_{\As_{pq}^s(\R^n)}+\sum_{j\in J_+}2^{\delta|j|}|a_j|b_j^{m\cdot\widehat M_{pq}^{ns}}b_j^{M_{pq}^{ns}}\|\tilde f_j-\tilde f_-\|_{\As_{pq}^s(\R^n)}&(\text{by Proposition~\ref{Prop::Dila}})
    \\
    \lesssim&_{\min(1,p,q)}\|f\|_{\As_{pq}^s(B_+(0,R))}+\sum_{j\in J_+}2^{\delta|j|}|a_j|b_j^{m\cdot\widehat M_{pq}^{ns}+M_{pq}^{ns}}\big(\|f\|_{\As_{pq}^s(B_+(0,R))}+\|f\|_{\As_{pq}^s(B_+(0,b_jR))}\big).\hspace{-2in}
\end{align*}
Here $C'=C'_{\delta, a, b,\widehat M}=\sum_{j\in\Z}2^{\delta|j|}|a_j|\big(b_j^{\widehat M_{pq}^{ns}}+b_j^{-\widehat M_{pq}^{ns}}\big)$.
Notice that all implied constants in the above computation do not depend on the $R>0$.

Therefore for general $R>0$, applying \eqref{Eqn::CondAB} again with a small $\delta$, we see that for every extension $\tilde f$ of $f$,
\begin{align*}
    &R^{-n\tau}\|Ef\|_{\As_{pq}^s(B(0,R))}\le R^{-n\tau}\|\tilde F^R\|_{\As_{pq}^s(\R^n)}
    \\
    \lesssim&_{p,q,s,m,\delta}C'_{\delta a b\widehat M}\cdot R^{-n\tau}\|f\|_{\As_{pq}^s(B_+(0,R))}+\sum_{j\in J_+}2^{\delta|j|}|a_j|b_j^{C''_{npqs\tau m}}(R^{-n\tau}\|f\|_{\As_{pq}^s(B_+(0,R))}+(b_jR)^{-n\tau}\|f\|_{\As_{pq}^s(B_+(0,b_jR))})
    \\
    \lesssim&_{a,b,n,\tau,m,\widehat M,M,\delta}\sup_{\rho\in\{1,b_j:b_j>1\}}\rho^{-n\tau}\|\tilde f\|_{\As_{pq}^s(B(0,\rho))}\lesssim_{n,p,q,s,\tau}\|\tilde f\|_{\As_{pq}^{s\tau}(\R^n)}.
\end{align*}
Here $C''=C''_{n,p,q,s,\tau,m}:=m\cdot\widehat M_{pq}^{ns}+M_{pq}^{ns}+n\tau$ and we let $\delta>0$ be such that $\sum_{j\in\Z}2^{\delta|j|}|a_j|b_j^{\lceil C''\rceil}<\infty$. Taking the infinmum over the extension $\tilde f$, we obtain \eqref{Eqn::PfMorreyBdd::Key} and conclude the proof of the case $s<0$.

\smallskip For $s\ge0$, we use the following equivalent norm (see \cite[Theorem~1.6]{WuYangYuanDer}):
\begin{equation}\label{Eqn::EqvNormDer}
    \|\tilde f\|_{\As_{pq}^{s\tau}(\R^n)}\approx_{n,p,q,s,\tau,m}\sum_{|\gamma|\le m}\|\partial^\gamma \tilde f\|_{\As_{pq}^{s-m,\tau}(\R^n)},\quad\forall\, m\ge1.
\end{equation}
Using Lemma~\ref{Lem::Trivial} \ref{Item::Trivial::Comm}, and taking $m\ge s+1$ we have
\begin{equation}\label{Eqn::PfMorreyBdd::DerArg}
    \|E^{a,b}f\|_{\As_{pq}^{s\tau}(\R^n)}\approx\sum_{|\gamma|\le m}\|E^{a(-b)^{\gamma_n},b}\partial^\gamma f\|_{\As_{pq}^{s-m,\tau}(\R^n)}\lesssim\sum_{|\gamma|\le m}\|\partial^\gamma f\|_{\As_{pq}^{s-m,\tau}(\R^n_+)}\lesssim\|f\|_{\As_{pq}^{s\tau}(\R^n_+)}.
\end{equation}
This completes the proof of $E:\As_{pq}^{s\tau}(\R^n_+)\to\As_{pq}^{s\tau}(\R^n)$ for $\As\in\{\Bs,\Fs\}$ and all $0<p,q\le \infty$, $s\in\R$ and $0\le \tau\le\frac1p$ ($p<\infty$ if $\As=\Fs$). The case $\As=\Ns$ follows from real interpolations, as mentioned in the beginning of the proof.
\end{proof}
% Finally, the boundedness of $E^*$ follows from the similar argument to \eqref{Eqn::PfMorreyBdd::DilSum}. For $g\in\Ss'(\R^n)$, $g+\sum_j\frac{a_j}{b_j}\vartheta^{-1/b_j}g$ is an extension of $E^*g$. Using \eqref{Eqn::EqvNormDer} with $m>s$ and Proposition~\ref{Prop::Stuff} \ref{Item::Stuff::FInftyQ} we have 
% \begin{align*}
%     &\|E^*g\|_{\As_{pq}^{s\tau}(\R^n_+)}\le\textstyle\|g+\sum_j\frac{a_j}{b_j}\cdot\vartheta^{-1/b_j}g\|_{\As_{pq}^{s\tau}(\R^n)}\approx\sum_{|\gamma|\le m}\big\|\partial^\gamma g-\sum_ja_j(-b_j)^{-1-\gamma_n}\cdot\vartheta^{-1/b_j}[\partial^\gamma g]\big\|_{\As_{pq}^{s-m,\tau}(\R^n)}
%     \\
%     \lesssim&\sum_{|\gamma|\le m}\sup_{x\in\R^n,R>0}R^{-n\tau}\Big(\|\partial^\gamma g\|_{\As_{pq}^{s-m}(B(x,R))}+\sum_j2^{\delta|j|}\tfrac{|a_j|}{b_j^{\gamma_n+1}}\|\vartheta^{-1/b_j}\partial^\gamma g\|_{\As_{pq}^{s-m}(B(x,R))}\Big)
%     \\
%     \lesssim&\sum_{|\gamma|\le m}\sup_{x\in\R^n,R>0}R^{-n\tau}\Big(\|\partial^\gamma g\|_{\As_{pq}^{s-m}(B(x,R))}+\sum_j2^{\delta|j|}\tfrac{|a_j|}{b_j^{\gamma_n+1}}\big(b_j^{-M_{pq}^{ns}}+b_j^{M_{pq}^{ns}}\big)\|\partial^\gamma g\|_{\As_{pq}^{s-m}\big(B(\vartheta^{-b_j}x,\max(R,b_jR))\big)}\Big)
%     \\
%     \lesssim&\sum_{|\gamma|\le m}\sup_{y\in\R^n,\rho>0}\rho^{-n\tau}\|\partial^\gamma g\|_{\As_{pq}^{s-m}(B(y,\rho))}\approx\sum_{|\gamma|\le m}\|\partial^\gamma g\|_{\As_{pq}^{s-m,\tau}(\R^n)}\approx\|g\|_{\As_{pq}^{s\tau}(\R^n)}.
% \end{align*}

% The proof is now complete.

\begin{rem}
    By keeping track of the constants, and possibly enlarging the constant $\widehat M=\widehat M_{p,q}^{n,s}$ in Theorem~\ref{Thm::QntBFBdd}, it is likely to prove that \eqref{Eqn::QntBFBdd::ENorm} is also true if we replace $\As_{pq}^s$ by $\As_{pq}^{s\tau}$ for every $0\le \tau\le\frac1p$ and $\As\in\{\Bs,\Fs,\Ns\}$ ($p<\infty$ for $\Fs$-cases).
\end{rem}

\section{The Analogy on Bounded Smooth Domains}
On bounded smooth domains the analogy of Theorem~\ref{MainThm} is the following:
\begin{thm}[Universal Seeley's extension on smooth domains]\label{Thm::Domain}
Let $\Omega\subset\R^n$ be a bounded smooth domain with a smooth defining function\footnote{A defining function of $\Omega$ is a real function (at least $C^1$) $\sigma:\R^n\to\R$ such that $\Omega=\{x:\sigma(x)<0\}$ and $\nabla\sigma|_{\partial\Omega}\neq0$.} $\sigma$, and let $\Uc\subset\R^n$ be a bounded neighborhood of $\overline\Omega$. Then there exists an extension operator $\Ec:\Ss'(\Omega)\to\Es'(\Uc)$ that maps extendable distributions in $\Omega$ (recall it from Remark~\ref{Rmk::DefSs} \ref{Item::DefSs::BddDom}) to compactly supported distributions in $\Uc$, such that
\begin{enumerate}[(i)]
    \item{\normalfont(One-parameter dependence)}\label{Item::Domain::Dep} For every $x\in \Uc\backslash\overline\Omega$ and $f\in C^0(\overline\Omega)$, the value $\Ec f(x)$ depends only on the value of $f$ on the (unique) integral curve for $\nabla\sigma$ passing through $x$ that contains in $\Omega$, which is, the set 
    $$\{\gamma^x(t):t\in\R,\, \gamma^x(t)\in\Omega\}\subset \Uc,\text{ where }\dot\gamma^x(t)=\nabla\sigma(\gamma^x(t))\ \forall t\in\R\text{ and }\gamma^x(0)=x.$$
    
    \item{\normalfont(Universal boundedness)}\label{Item::Domain::Bdd} $\Ec:W^{k,p}(\Omega)\to W^{k,p}_c(\Uc)$ and $\Ec:C^{k,s}(\Omega)\to C^{k,s}_c(\Uc)$ are defined and bounded for all $k\in\Z$, $0<p\le\infty$ and $0<s<1$. In particular $\Ec:C^k(\Omega)\to C^k_c(\Uc)$ for $k\ge0$.
    
    $\Ec:\As_{pq}^{s\tau}(\Omega)\to \As_{pq}^{s\tau}(\Uc)$ is bounded for $\As\in\{\Bs,\Fs,\Ns\}$, $p,q\in(0,\infty]$, $s\in\R$  and $\tau\in[0,\frac1p]$ ($p<\infty$ for $\Fs$-cases). In particular $\Ec:\As_{pq}^s(\Omega)\to \As_{pq}^s(\Uc)$ is bounded for $\As\in\{\Bs,\Fs\}$, $0<p,q\le\infty$ and $s\in\R$.
\end{enumerate}
\end{thm}

Note that \ref{Item::Domain::Dep} is false if one use the usual patching construction, see, for example \cite[(3.3.4/6)]{TriebelTheoryOfFunctionSpacesI}.

In the proof we need the standard partition of unity argument.
\begin{lem}\label{Lem::PartUnity} Let $\Phi:\R^n\to\R^n$ be a bounded diffeomorphism, i.e. $\Phi$ is bijective and
\begin{equation}\label{Eqn::BddDiffDef}
    \sup_{x\in\R^n}|\partial^\alpha\Phi(x)|+|\partial^\alpha(\Phi^{-1})(x)|<\infty,\quad\forall\,\alpha\in\N_0^n\backslash\{0\}.
\end{equation}

Let $\chi\in C_c^\infty(\R^n)$ be a bounded smooth function. Define a linear map $Tf:=(\chi\cdot f)\circ\Phi$.
\begin{enumerate}[(i)]
    \item\label{Item::PartUnity::Sob} $T:W^{k,p}(\R^n)\to W^{k,p}(\R^n)$, $T:C^{k,s}(\R^n)\to C^{k,s}(\R^n)$ are both defined and bounded for all $k\in\Z$, $0<p\le\infty$ and $0<s<1$.
    In particular $T:C^k(\R^n)\to C^k(\R^n)$ for $k\in\N_0$.
    \item\label{Item::PartUnity::BF} $T:\As_{pq}^{s\tau}(\R^n)\to\As_{pq}^{s\tau}(\R^n)$ for $\As\in\{\Bs,\Fs,\Ns\}$, $p,q\in(0,\infty]$, $s\in\R$ and $\tau\in[0,\frac1p]$ ($p<\infty$ if $\As=\Fs$). 

In particular $T:\As_{pq}^s(\R^n)\to\As_{pq}^s(\R^n)$ is bounded for $\As\in\{\Bs,\Fs\}$, for $p,q\in(0,\infty]$ and $s\in\R$.
\end{enumerate}
\end{lem}
\begin{proof}The result \ref{Item::PartUnity::Sob} follows from using chain rules, product rules and the formula of change of variables:
\begin{equation*}
    \int_{\R^n}|(\chi f)\circ\Phi(x)|^pdx=\int_{\R^n}|\chi(y) f(y)|^p|\det\nabla(\Phi^{-1})(y)|dy,\quad 0<p<\infty,\quad f\in L^p_\loc(\R^n).
\end{equation*}
Notice that $|\det\nabla(\Phi^{-1})|$ is a bounded smooth function, and all nonzero derivatives of $\Phi$ and $\chi$ are bounded. We leave the details to readers.

For \ref{Item::PartUnity::BF}, see for example \cite[Sections 6.1.1 and 6.2]{YSYMorrey} and \cite[Sections 4.2 and 4.3]{HaroskeTriebelMorrey}.
\end{proof}

\begin{proof}[Proof of Theorem~\ref{Thm::Domain}]In the following, we let $\Xs$ be the function class that belongs to the following:
\begin{multline*}
    \{W^{k,p},C^{k,s}:k\in\Z,\, p\in(0,\infty],\,s\in(0,1)\}\cup\{\Bs_{pq}^s,\Fs_{pq}^s,\Bs_{pq}^{s\tau},\Ns_{pq}^{s\tau}:p,q\in(0,\infty],\, s\in\R,\,\tau\in[0,\tfrac1p]\}
    \\\cup\{C^k:k\in\N_0\}\cup\{\Fs_{pq}^{s\tau}:p\in(0,\infty),\,q\in(0,\infty],\, s\in\R,\,\tau\in[0,\tfrac1p]\}.
\end{multline*}
Recall from Theorem~\ref{MainThm} that $E^{a,b}:\Xs(\R^n_+)\to\Xs(\R^n)$ is bounded for all such $\Xs$.

\smallskip
For a function $g(\theta,t)$ defined on $\partial\Omega\times\R_+$, we define its extension $\Ec_0g$ on $\partial\Omega\times\R$ by
\begin{equation}\label{Eqn::DefEc0}
    \Ec_0g(\theta,t):=\begin{cases}\sum_{j=-\infty}^\infty a_j\cdot g(\theta,-b_jt)&t<0;\\g(\theta,t)&t>0
    \end{cases},\quad\text{where }(a_j,b_j)_{j=-\infty}^\infty\text{ satisfies }\eqref{Eqn::CondAB}.
\end{equation}
Recall from Proposition~\ref{Prop::ConsAB} that such $(a_j,b_j)_j$ exist.

\smallskip
Since $\nabla\sigma$ is nonzero in a neighborhood of $\partial\Omega$, we see that there is a small $\eps_0>0$, such that the ODE flow map $\exp_{\nabla\sigma}:\partial\Omega\times(-\eps_0,\eps_0)\to\R^n$ given by 
\begin{equation*}
    \exp_{\nabla\sigma}(\theta,0)=\theta,\qquad\tfrac\partial{\partial t}\exp_{\nabla\sigma}(\theta,t)=\nabla\sigma(\Psi(\theta,t)),\qquad\text{for }\theta\in\partial\Omega,\quad -\eps_0<t<\eps_0,
\end{equation*}
is defined and is diffeomorphic onto its image. Shrinking $\eps_0$ if necessary we can assume that $\exp_{\nabla\sigma}(\partial\Omega\times(-\eps_0,\eps_0))\subseteq\Uc$. 

Since $\frac\partial{\partial t}\exp_{\nabla\sigma}(\theta,t)=\nabla\sigma(\Psi(\theta,t))$ and $\sigma(\exp_{\nabla\sigma}(\theta,0))=0$ we see that $\sigma(\exp_{\nabla\sigma}(\theta,t))=t$ hence $\{x:-\eps_0<\sigma(x)<\eps_0\}\subseteq\Uc$. 

We define $\Psi:\{-\eps_0<\sigma<\eps_0\}\to\partial\Omega\times(-\eps_0,\eps_0)$ by $\Psi:=(\exp_{\nabla\sigma})^{-1}$. Since $\sigma(\exp_{\nabla\sigma}(\theta,t))=t$, we see that
\begin{equation}\label{Eqn::Domain::Psi}
    \Psi(x)=(\psi'(x),\sigma(x))\text{ for some smooth map }\psi':\{-\eps_0<\sigma<\eps_0\}\to\partial\Omega.
\end{equation}

We take the following cutoff functions:
\begin{equation*}\label{Eqn::Domain::Chi}
    \chi_0,\chi_1\in C_c^\infty(-\tfrac23\eps_0,\tfrac23\eps_0)\text{ such that }\chi_0|_{[-\frac{\eps_0}2,\frac{\eps_0}2]}\equiv1,\ \chi_1|_{\supp\chi_0}\equiv1;\quad\text{and}\quad\tilde \chi_0:=\chi_0\circ\sigma,\ \tilde\chi_1:=\chi_1\circ\sigma.
\end{equation*}

We have $\tilde\chi_0,\tilde \chi_1\in C_c^\infty(\Uc)$. We define $\Ec$ as 
\begin{equation}\label{Eqn::DefEc}
    \Ec f:=(1-\tilde\chi_0)|_\Omega\cdot f+\tilde\chi_1\cdot(\Ec_0[(\tilde\chi_0\cdot f)\circ\Psi^{-1}]\circ\Psi).
\end{equation}

Clearly $\Ec f|_\Omega=(1-\tilde\chi_0)\cdot f+\tilde\chi_1\cdot\tilde\chi_0\cdot f=f$. Therefore $\Ec$ is an extension operator. 

From \eqref{Eqn::Domain::Psi} we see that the pullback vector field $\Psi^*\Coorvec t$ is exactly $\nabla\sigma$ in the domain. Since for each $(\theta,t)$, the value $\Ec_0 g(\theta,t)$ depends only on the line $\{g(\theta,s):s\in\R\}$, taking pullback we get \ref{Item::Domain::Dep}.

% To prove \ref{Item::Domain::Dep}, if a vector field $X$ satisfies $X\cdot\nabla\sigma=0$, then by \eqref{Eqn::Domain::Psi}, $(\Psi_*X)(\theta,t)\in T_\theta(\partial\Omega)$ for every $(\theta,t)$ in the domain. Therefore $\Psi_*X$ is a differential operator that only acts on $\theta$ but not $t$. By the formula \eqref{Eqn::DefEc0} we see that $\nabla_\theta\circ\Ec_0=\Ec_0\circ\nabla_\theta$

\medskip
We postpone the proof of $\Ec:\Ss'(\Omega)\to\Es'(\Uc)$ to the end. It is more convenient to prove \ref{Item::Domain::Bdd} by passing to local coordinate charts.

We consider the following objects:
\begin{itemize}
    \item Take a coordinate cover $\{\phi_\lambda':V_\lambda\subseteq\partial\Omega\to\R^{n-1}\}_{\lambda=1}^N$ of $\partial\Omega$, where $V_\lambda$ are open in the submanifold $\partial\Omega$.
    \item For $1\le \lambda\le N$, define $\Phi_\lambda:V_\lambda\times(-\eps_0,\eps_0)\to\R^{n-1}\times\R^1$ by $\Phi_\lambda(\theta,t):=(\phi'_\lambda(\theta),t)$.
    \item Take  a partition of unity $\{\mu_\lambda\in C_c^\infty(V_\lambda)\}_{\lambda=1}^N$ of $\partial\Omega$, i.e. $\sum_{\lambda=1}^N\mu_\lambda\equiv1$ on $\partial\Omega$.
    \item Let $\{\rho_\lambda\in C_c^\infty(V_\lambda)\}_{\lambda=1}^N$ be such that $\rho_\lambda|_{\supp\mu_\lambda}\equiv1$.
    \item For $1\le \lambda\le N$, we let 
    $$\tilde\mu_\lambda:=(\mu_\lambda\otimes\chi_0)\circ\Psi=(\mu_\lambda\circ\psi')\cdot(\chi_0\circ\sigma)\quad\text{and}\quad \tilde\rho_\lambda(x):=(\rho_\lambda\otimes\chi_1)\circ\Phi_\lambda^{-1}(x)=(\rho_\lambda\circ{\phi'_\lambda}^{-1})(x')\cdot\chi_1(x_n).$$
\end{itemize}
See the following picture:

{
\centerline{
\includegraphics[width=0.9\textwidth]{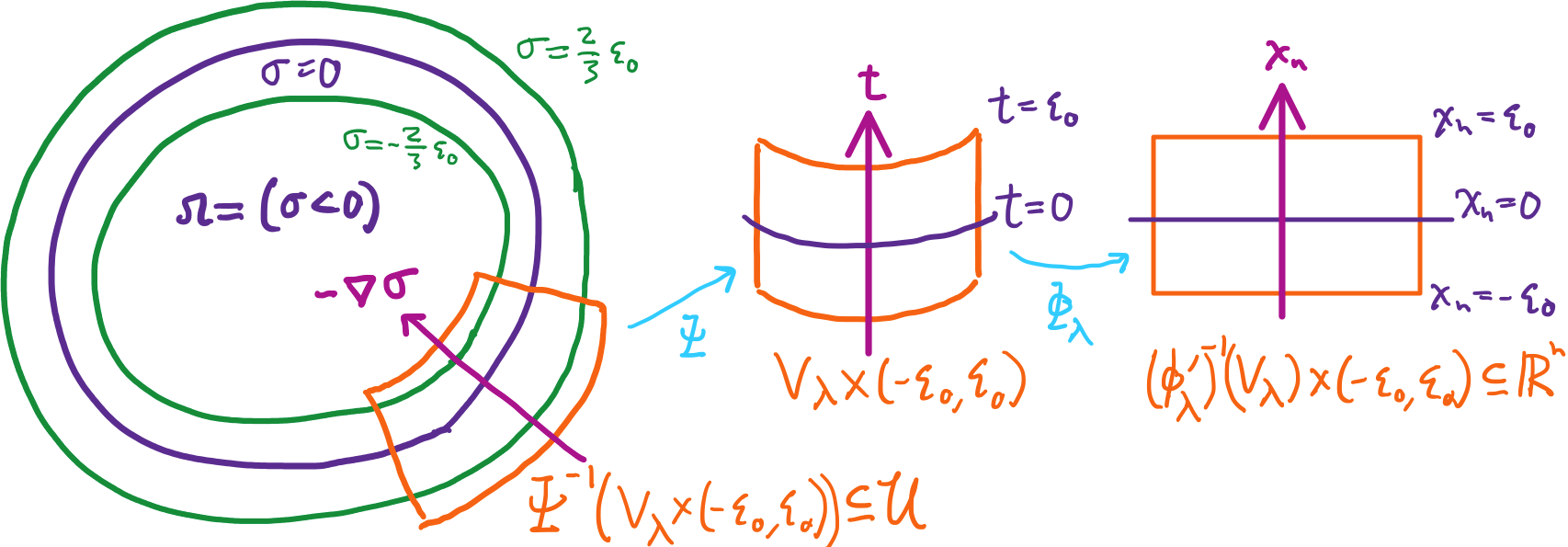}}
}

Now $\{\Phi_\lambda\circ\Psi:\Psi^{-1}(V_\lambda\times(-\eps_0,\eps_0))\to\R^n\}_\lambda$ is a coordinate cover of $\{-\eps_0<\sigma<\eps_0\}$. By shrinking each $V_\lambda$ if necessary, we can ensure the following:
\begin{itemize}
    \item For every $1\le\lambda\le N$, $\Phi_\lambda\circ\Psi:\Psi^{-1}\big(V_\lambda\times(-\frac23\eps_0,\frac23\eps_0)\big)\subseteq\Uc\to\R^n$ can extend to a bounded diffeomorphism. In other words, there are $\Xi_\lambda:\R^n\to\R^n$ that  satisfies \eqref{Eqn::BddDiffDef} and $\Xi_\lambda|_{\Psi^{-1}(V_\lambda\times(-\frac23\eps_0,\frac23\eps_0)}=\Phi_\lambda\circ\Psi$, for each $\lambda$.
\end{itemize}

From the construction of $\Psi$, $\Phi_\lambda$, $\mu_\lambda$ and $\rho_\lambda$ we see that
\begin{itemize}
    \item $\tilde\mu_\lambda\in C_c^\infty(\Psi^{-1}(V_\lambda\times(-\eps_0,\eps_0)))$ and $\tilde\rho_\lambda\in C_c^\infty(\Phi_\lambda(V_\lambda\times(-\eps_0,\eps_0)))$.
    \item When $\supp g\subseteq(\supp\mu_\lambda)\times[0,\infty)$, we have $\supp \Ec_0g\subseteq(\supp\mu_\lambda)\times\R$. 
    
    In particular $(\1_{\partial\Omega}\otimes\chi_1)\cdot\Ec_0[(\mu_\lambda\otimes\chi_0)\cdot g]=(\rho_\lambda\otimes\chi_1)\cdot\Ec_0[(\mu_\lambda\otimes\chi_0)\cdot g]$ (since $\rho_\lambda|_{\supp\mu_\lambda}\equiv1$).
    \item When $\supp g\subseteq V_\lambda\times[0,1)$, we have $\Ec_0[g]|_{V_\lambda\times(-1,1)}=E[g\circ\Phi_\lambda^{-1}]\circ\Phi_\lambda$, the right hand side of which is well-defined. Here $E$ is the extension operator from \eqref{Eqn::ExtOp}.
\end{itemize}

Therefore
\begin{equation}\label{Eqn::Domain::Part}
    \begin{aligned}
    &\tilde\chi_1\cdot(\Ec_0[(\tilde\chi_0\cdot f)\circ\Psi^{-1}]\circ\Psi)=\sum_{\lambda=1}^N\big((\1_{\partial\Omega}\otimes\chi_1)\cdot \Ec_0[(\tilde\mu_\lambda\cdot f)\circ\Psi^{-1}]\big)\circ\Psi
    \\
    =&\sum_{\lambda=1}^N\left((\rho_\lambda\otimes\chi_1)\cdot\big( E[(\tilde\mu_\lambda\cdot f)\circ\Psi^{-1}\circ\Phi_\lambda^{-1}]\circ\Phi_\lambda\big)\right)\circ\Psi=\sum_{\lambda=1}^N\left(\tilde\rho_\lambda\cdot E[(\tilde\mu_\lambda\cdot f)\circ(\Phi_\lambda\circ\Psi)^{-1}]\right)\circ(\Phi_\lambda\circ\Psi).
\end{aligned}
\end{equation}

Applying Lemma~\ref{Lem::PartUnity} and Theorem~\ref{MainThm} \ref{Item::Thm::SobBdd} - \ref{Item::Thm::MorreyBdd}, we have the boundedness of compositions
\begin{equation*}
    \Xs(\Omega)\xrightarrow{f\mapsto\left((\tilde\mu_\lambda f)\circ(\Phi_\lambda\circ\Psi)^{-1}\right)_{\lambda=1}^N}\Xs(\R^n_+)^N\xrightarrow{E}\Xs(\R^n)^N\xrightarrow{(g_\lambda)_\lambda\mapsto\sum_{\lambda=1}^N (\tilde\rho_\lambda g)\circ(\Phi_\lambda\circ\Psi)}\Xs_c(U_0)\subset\Xs_c(\Uc).
\end{equation*}

On the other hand, $\tilde\chi_0=\chi_0\circ\rho\in C_c^\infty(U_0)$ is identical to $1$ near $\partial\Omega$, therefore $(1-\tilde\chi_0)|_\Omega\in C_c^\infty(\Omega)$. By Lemma~\ref{Lem::PartUnity} $[f\mapsto (1-\tilde\chi_0)|_\Omega\cdot f]:\Xs(\Omega)\to\Xs_c(\Uc)$ is bounded. Altogether we conclude that $\Ec:\Xs(\Omega)\to\Xs(\Uc)$ is bounded. This completes the proof of \ref{Item::Domain::Bdd}.

\smallskip
Finally we illustrate the definedness of $\Ec:\Ss'(\Omega)\to\Es'(\Uc)$. Note that alternatively one can obtain this statement by using duality and prove the boundedness $\Ec^*:C_\loc^\infty(\Uc)\to\Ss(\Omega)(=\{g\in C^\infty(\R^n):\supp g\subseteq\overline{\Omega}\})$.

For a $f\in\Ss'(\Omega)$, let $\tilde f$ be an extension of $f$. Since $\Omega$ is a bounded domain, we can assume $\tilde f$ to have compact support, otherwise we replace $\tilde f$ with $\kappa\tilde f$ where $\kappa\in C_c^\infty(\R^n)$ is such that $\kappa|_{\Omega}\equiv1$.

By the structure theorem of distributions (see \cite[Theorem~6.27]{GrandpaRudin} for example), there are $M\ge0$ and $\{g_\alpha\}_{|\alpha|\le M}\subset C_c^0(\R^n)$ such that $\tilde f=\sum_{|\alpha|\le M}\partial^\alpha g_\alpha$. By Proposition~\ref{Prop::Stuff} \ref{Item::Stuff::SumDer} this is saying that $\tilde f\in\Bs_{\infty\infty}^{-M}(\R^n)$ for some $M>0$, which means $f\in\Bs_{\infty\infty}^{-M}(\Omega)$ for the same $M$.

Now by \ref{Item::Domain::Bdd} $\Ec f\in\Bs_{\infty\infty}^{-M}(\Uc)\subset\Ds'(\Uc)$ has compact support, we conclude that $\Ec f\in\mathscr E'(\Uc)$.
\end{proof}

% \section{Further Remark}
% Boundedness of Gevrey class?

\vspace{0in}
\bibliographystyle{amsalpha}
\bibliography{reference} 
\end{document}